\documentclass[a4paper, 12pt]{article}

\usepackage[utf8]{inputenc}
\usepackage{ amssymb, amsthm }
\usepackage{ amsmath }
\usepackage{ hyperref }
\usepackage{ tikz }
\usepackage{microtype}
\DeclareMathOperator{\conv}{conv}

\DeclareMathOperator{\hess}{Hess}
\DeclareMathOperator{\newt}{Newt}

\DeclareMathOperator{\vol}{V}
\DeclareMathOperator{\len}{len}
\DeclareMathOperator{\val}{val}
\DeclareMathOperator{\trop}{trop}
\DeclareMathOperator{\cone}{cone}

\newtheorem{theorem}{Theorem}[section]
\newtheorem{lemma}[theorem]{Lemma}
\newtheorem{proposition}[theorem]{Proposition}
\newtheorem{corollary}[theorem]{Corollary}
\newtheorem{remark}[theorem]{Remark}
\theoremstyle{definition}
\newtheorem{definition}[theorem]{Definition}
\newtheorem{example}[theorem]{Example}
\newtheorem{assumption}{Assumption}

\author{Aliaksandr Yuran \thanks{HSE university. The author is partially supported by International Laboratory of Cluster Geometry NRU HSE, RF Government grant, ag. № 075-15-2021-608 from 08.06.2021}}
\title{Plucker Formulas for Plane Algebraic Curves with a Given Newton Polygon }
\begin{document}
\maketitle 
	\begin{abstract}
		
		Let $C$ be a generic complex plane curve with a given Newton polygon $P$. We compute the number of its inflection points and bitangents (equivalently, the number of singularities of the projectively dual curve $C^\vee$). We also prove that $C^\vee$ has no singularities other than nodes and cusps for large enough
		polygons $P$.
	\end{abstract}
	%\maketitle
	\newcommand {\torus}{(\mathbb C \setminus 0)}
	\newcommand {\CC}{\mathbb C}
	\newcommand {\ol}{\overline}
	\newcommand {\da}{{\downarrow}}
	\newcommand{\ua}{{\uparrow}}
	\newcommand{\swa}{{\swarrow}}
	\newcommand{\nea}{{\nearrow}}
	\newcommand{\la}{{\leftarrow}}
	\newcommand{\ra}{{\rightarrow}}
	\newcommand{\initver}{\tikzstyle{ver}=[circle,draw=black,fill=black!40, thin, inner sep=0pt,minimum size=0.2cm]}

	\section{Introduction}
	
	Let $C$ be a generic complex plane projective curve of degree $d>1$. The classical Pl\"ucker formulas (see e.g. \cite{ghbook}) compute the number of inflection points and bitangents of $C$. They are 
	\begin{align*}
		&\#(\mathrm{infl}) = 3 d (d-2);\\
		&\#(\mathrm{bitang}) = {1 \over 2} d (d+3) (d-3) (d-2).
	\end{align*}
	We prove their analogue for a generic curve $C \subset \torus^2$ with given Newton polygon $P$. There is a standard embedding $i\colon \torus^2 \hookrightarrow \mathbb P^2$ given by 
	\begin{equation}
		i \colon (a,b) \mapsto (a:b:1).
	\end{equation}

	The closure of the image $\ol{i(C)}$ is a projective curve whose singular locus is contained in the set $\{(1:0:0),(0:1:0), (0:0:1)\}$. We induce the notions of bitangents and inflection from $\mathbb P^2$. A bitangent is a line that is tangent to $i(C)$ at two distinct points that lie in $\torus^2$. A point of $i(C)$ is an inflection point if the tangent at it intersects the curve with multiplicity 3. (In fact, we could allow tangency at any nonsingular point of the closure $\ol{i(C)} \subset \mathbb P^2$, but for some polygons $P$ we would obtain a different answer.) 
	
	%The number of inflection points for $P$ that do not have edges of certain directions was computed by Erwan Brugallé and Lucía López de Medrano in \cite[Theorem 6.1]{tropicalinflection}. They also provide a bound for this number from above for all $P$. We show that the bound is sharp for large enough polygons.
	
	%Bitangents to plane quartics were studied by tropical methods by ,,,,,, in ,,,,,,
	
	The papers \cite{infltrop}, \cite{liftbitangents}, \cite{combliftbitangents} and \cite{gen0} allow to study valuations of the flexes and bitangents that we count, under a tropical degeneration of the plane curve. Ultimately, it would be interesting to combinatorially interpret the answers that we obtain as the (weighted) numbers of flexes and bitangent lines of a generic tropical curve with a given Newton polygon.
	
	In particular, \cite[Theorem 6.1]{infltrop} gives a bound on the number of flexes from above. We prove that it is sharp for most polygons.
	
	There are other possible notions of inflection, e.g. log-inflection points that are considered in \cite{loginflection}. These are the points where a \textit{logarithmic line} intersects $C$ with multiplicity 3 rather than a usual line. 
	
	\paragraph{Acknowledgement.}The author is grateful to professor Alexander Esterov for constant attention to this work.
	%\paragraph{Structure of the paper.} In the introduction we set up the notation and formulate main results (Theorems \ref{inflection} and \ref{bitangents}) and the assumptions on $P$ under which the theorems hold, and also recall basic facts about Newton polygons and toric varieties.
	
	%In Sections 2 and 3 we prove the two theorems mentioned above; in Section 4 we give sufficient conditions on $P$ under which the assumptions hold.
	
	\subsection{Main definitions and results}
	
	\begin{definition}
		Let  $f=\sum_{k\in \mathbb Z^n}  a_k \mathbf z^k$ be a Laurent polynomial of $n$ complex variables where $\mathbf z^k=z_1^{k_1} z_2^{k_2} \dots z_n^{k_n}$. The Newton diagram of $f$ is the set $\mathcal P=\{k \in \mathbb Z ^n \mid a_k \ne 0\}$. The Newton polytope $\newt (f)$ is the convex hull of $\mathcal P$. 
		
		The Newton polytope of a hypersurface in $\torus^n$ is the Newton polytope of its defining equation. It is defined up to a lattice translation.
		
		Let $\mathcal P$ be any subset of $\mathbb R ^n$. By $\CC^\mathcal P$ denote the space of all Laurent polynomials that can be expressed as $\sum_{k\in \mathbb Z^n \cap \mathcal P}  a_k \mathbf z^k$. 
	\end{definition}
	
	\begin{definition}
		Let $\gamma \in \mathbb (\mathbb R^n)^*$ be a linear function and $P \subset \mathbb R^n$ be a bounded set. The support set $P^\gamma$ is the subset of $P$ on which $\gamma|_P$ attains its \textit{maximal} value. 
		
	\end{definition}
	
	We denote the linear functions $(x,y) \mapsto -y$, $(x,y) \mapsto x+y$, $(x,y) \mapsto -x$ by $\da$, $\nea$, $\la$ respectively. Set also $\ua = - \da$, $\swa = -\nea$, $\ra = -\la$.

	By $\vol(P)$ denote the area of a polygon $P$. By $\len(AB)$ denote the lattice length of the segment $AB$ (i.e. $\#(AB \cap \mathbb Z^n) - 1$). In particular, $\len(\text{point})=0$.  
	
	Now we are ready to formulate the main results of the paper. 	
	
	By $\Delta$ denote the standard triangle with vertices $(0,0)$, $(1,0)$, $(0,1)$. A line is the zero locus of a polynomial from $\CC^\Delta$.
	
	Let $C\subset \torus^2$ be a curve and $l\subset \torus^2$ be a line. We say that $l$ is a simple tangent if $l$ is tangent to $C$ at exactly one point $p$ such that multiplicity of intersection of $C$ and $l$ at $p$ is $C \circ_p l=2$. The line $l$ is a simple inflection tangent if $l$ is tangent to $C$ at exactly one point $p$ with $C\circ_pl=3$. In this case $p$ is an inflection point. The line $l$ is a simple bitangent if $l$ is tangent to $C$ at exactly two points $p$, $q$ with $C\circ_p l = C\circ_q l = 2$.
	
	\begin{theorem} \label{mnthm}
		Assume that the polygon $P$ contains the dilated standard triangle $5\Delta$. Consider a generic curve $C$ with Newton polygon $P$. Then:
		
		1) Any line is either transversal to $C$, or is a simple tangent, or is a simple inflection tangent, or is a simple bitangent.
		
		2) The number of inflection points of $C$ is 	\begin{equation} \label{mninflections}
			6\vol(P)-2\sum_{ \gamma\in \da,\nea,\la}\len(P^\gamma)-\sum_{\gamma\in\ua,\swa,\ra}\len(P^\gamma).
		\end{equation}
	
		3) The number of bitangents is \begin{equation} \label{mnbitangents}
		-10\vol(P) + \vol(P^\vee) + 3\sum_{\gamma\in\da,\nea,\la}\len(P^\gamma) + \sum_{\gamma\in\ua,\swa,\ra}\len(P^\gamma) .
		\end{equation} 
		
		Here $P^\vee$ is a certain polygon that we describe below. Its area is explicitly computed in Proposition \ref{dualvolume}
	\end{theorem}

In the following subsections we formulate more detailed Theorems \ref{inflection}--\ref{conditions} that cover the three parts of Theorem \ref{mnthm}. We prove these theorems in the main part of the text.

Throughout the paper $P$ is a lattice polygon on the plane. Suppose that $f$ is a generic polynomial in $\CC^P$. By $C$ we denote the curve $\{(x,y)\in \torus^2 \mid f(x,y)=0\}$. 

Just like in the classical Pl\"ucker problem, it is useful to consider the dual curve. Its nodes correspond to the bitangent lines to $C$ and cusps correspond to the inflection points of $C$.

\begin{definition}
	The closure of the set $\{(a,b)\in \torus^2 \mid \text{the line } ax + yb +1 = 0 \text{ is tangent to } C\}$ is called the dual curve and denoted by $C^\vee$. 
	
	By $P^\vee$ denote the Newton polygon of $C^\vee$.
\end{definition}
The polygon $P^\vee$ was described in \cite{dualvartrop}. We also give a self-contained proof of the following description in Definition \ref{deftrop} and Theorem \ref{dualtrop}. We compute $\vol(P^\vee)$ in Proposition \ref{dualvolume}.

\begin{theorem}
	Let $\gamma \in (\mathbb R^2)^*$ be a covector whose coordinates are coprime integers. Then
	\begin{equation*}
		\len(P^\vee)^\gamma=\begin{cases}
			 2\vol(P)-\len P^\gamma+\len P^{-\gamma} \text{, $\gamma \in \{\da, \nea, \la\}$;}\\
			 0\text{, $\gamma \in \{\ua, \swa, \ra\}$;}\\
			 \len P^{-\gamma} \text{, otherwise.}\\
		\end{cases}
	\end{equation*}
\end{theorem}

Here we specialize our formulas for the number of cusps and bitangents to some important Newton polygons.

\begin{example}
	Suppose that the faces $P^\da$, $P^\ua$, $P^\nea$, $P^\swa$, $P^\la$ and $P^\ra$ are vertices and a generic curve supported at $P$ has only simple inflections and bitangents. Then the number of them is as follows: 
	\begin{align*}
		&\#\mathrm{(infl)} = 6 \vol(P); \\
		&\#\mathrm{(bitang)} = \vol(P)(2\vol(P)+2\vol(\Delta,-P)-9);\\
		&P^\vee = 2\vol(P) \Delta + (- P).
	\end{align*}
	
	Here $-P$ denotes the polygon centrally symmetric to $P$ and the last addition is in the sense of Minkowski. The expression $\vol(\Delta,-P)$ stands for the mixed volume $\vol(\Delta + (-P))-\vol(\Delta)-\vol(P)$.
\end{example}

\begin{example}
	Suppose that $f$ is a generic bihomogeneous polynomial of bidegree $(c,d)$ (i.e. its Newton polygon $P$ is the rectangle with edges $c$ and $d$). 
	
	Then the curve $\{f=0\}$ has  \begin{align*}
		&\#\mathrm{(infl)} = 6cd-3c-3d;\\
		&\#\mathrm{(bitang)} = 2c^2d^2 - 10cd+4c+4d;\\ &P^\vee=2cd\Delta.\end{align*} 
\end{example}

\begin{example}
	Suppose that $f$ is a generic quasihomogeneous polynomial of degree $(c,d)$ with $c\ne d$ (i.e. its Newton polygon $P$ is the triangle with vertices $(0,0)$, $(c,0)$ and $(0,d)$). 
	
	Then the curve $\{f=0\}$ has  \begin{align*}
		&\#\mathrm{(infl)} = 3cd-2c-2d; \\ 
		&\#\mathrm{(bitang)} = {1\over2}(c^2d^2-11cd+6c+6d); \\ &P^\vee= \conv\{(c,0),(0,d),(0,cd),(cd,0)\}.
	\end{align*}
\end{example}

Theorem \ref{mnthm} asserts that the equalities \eqref{mninflections} and \eqref{mnbitangents} hold for the diagrams in the last two examples if $c \ge 5$ and $d \ge 5$. It is not hard to use the technique of Section \ref{assumptionsproof} to show that they hold for all $c$ and $d$ (see Example \ref{rectangle}).

\subsection{Toric varieties and tropical fans}

In this subsection we introduce the notation regarding toric varieties. 

Let $P\subset \mathbb R^n$ be a lattice polytope and $\{p_0, \dots, p_k\}$ be the set of its lattice points.

\begin{definition}
	The monomial map $m_P \colon \torus^n \rightarrow \mathbb{CP}^{k}$ is defined by $\mathbf z \mapsto (\mathbf z^{p_0}\colon\dots\colon \mathbf z^{p_k})$. The toric variety $X_P$ is the closure of the image of $m_P$. 	By definition, put $\mathcal O_{X_P} (P) = m_P^*(\mathcal O(1))$.
\end{definition}

%The natural action of the torus on $m_P(\torus^n)$ extends to an action on $X_P$. Its orbits are in one-to-one correspondence with faces of $P$. By $O

Each polynomial $\sum_{i = 0}^k a_i \mathbf z^{p_i} \in \CC^P$ corresponds to a section of $\mathcal O_{X_P} (P)$. Indeed, it can be viewed as the pullback $m_P^*(\sum_{i = 0}^k a_i z_i)$ of a homogenous linear polynomial on $\mathbb{CP}^{k}$.

Each face $Q \subsetneq P$ corresponds to an orbit of $X_P$. We denote it by $X_Q$.

\begin{definition}
	A 1-dimensional fan is a finite set of rational 1-dimensional cones (i.e. rays starting at the origin). 
	
	Suppose the rays $r_1, \dots, r_k$ form a fan. Assign a nonzero integer weight $a_i$ to each of the rays. Each ray contains a unique lattice vector $v_i$ of unit lattice length. The set of all pairs $(r_i, a_i)$ is called a tropical fan if $\sum_i a_i v_i = 0$. 
\end{definition}

\begin{definition}  \label{trop} \label{deftrop}
Let $C \subset \torus^2$ be a closed curve. Suppose that $P$ is a lattice polygon and that the compactification $\ol C \subset X_P$ does not intersect the $0$-dimensional orbits. Then the tropical fan of the curve $\trop C $ consists of the rays $r_Q = \{\gamma \in (\mathbb R^2)^* \mid P^\gamma = Q\}$ where $Q$ runs over  the set of edges of $P$. The weight of $r_Q$ is the index of intersection of $\ol C$ and the orbit $X_Q$. 
\end{definition}

It is known (see e.g. \cite[Chapter 6]{maclagan}) that $\trop C$ does not depend on the choice of polygon $P$. In particular, if $P$ is the Newton polygon of $C$, then the tropical fan coincides with the normal fan of $P$. The weights of the rays equal the lattice lengths of the corresponding edges of $P$.

	\subsection{The assumptions}
Throughout the paper we impose several assumptions on the polygon $P$ and prove formulas \eqref{mninflections} and \eqref{mnbitangents} of Theorem \ref{mnthm} under them.  

\begin{assumption} \label{assmain}
	The curve $C^\vee$ has only nodes (i.e. singularities of type $A_1$) and cusps (of type $A_2$) as singularities for the given $P$.
\end{assumption}

Consider the toric variety $X_P$ and the compactification $\ol C \subset X_P$.
\begin{definition} \label{definfleinfty}
	
	An inflection point of the curve $C$ at infinity is a point $p \in \ol C \cap X_{P^\gamma}$ where $\gamma \in \{\da, \la, \nea\}$ for which there exists a line $l\subset \torus^2$ such that the multiplicity of intersection $\ol C \circ_p \ol l \ge 3$.
\end{definition}

\begin{assumption} \label{assinfl}
	The curve $C$ does not have inflection points at infinity.
\end{assumption}

\begin{definition} \label{defasymptote}
	Suppose that $l\subset \torus^2$ is a vertical line.
	Then the curve $C$ tends asymptotically to $l$ as $y \rightarrow \infty$ if $\ol C$ and $\ol l$ have a common point on the orbit $X_{P^\ua}$. Suppose that $l$ is horizontal. Then the definition of an asymptote is the same, with $X_{P^\ua}$ is replaced by $X_{P^\ra}$.
	
	Suppose that $l$ passes through $(0,0)$. Then $l$ tends to $C$ as $x \rightarrow 0$ if $\ol l$ and $\ol C$ have a common point on $X_{P^\swa}$.
\end{definition}

\begin{assumption} \label{assinfty}
	Suppose that a vertical line $l$ is tangent to a generic curve $C$ with given Newton polygon $P$.  Then $l$ is neither a bitangent, nor a tangent at an inflection point, nor asymptotically tends to $C$ as $y \rightarrow \infty$. That is also the case for horizontal tangents (the asymptotes tend to $C$ as $x \rightarrow \infty$) and tangents that pass through the point $(0,0)$ (the asymptotes tend to $C$ as $x \rightarrow 0$).
\end{assumption}

So, the statements that we prove in Sections 2 and 3 are formulated as follows.

\begin{theorem} \label{inflection}
	Suppose the Assumptions 1 and 2 hold. Then the number of inflection points of the curve $C$ is given by formula \eqref{mninflections} of Theorem \ref{mnthm}.
\end{theorem}

\begin{theorem} \label{bitangents}
		Suppose the Assumptions 1, 2 and 3 hold. Then the number of bitangents to the curve $C$ is given by formula \eqref{mnbitangents} of Theorem \ref{mnthm}.
\end{theorem}

Finally, in Section 4 we prove the following theorem.

\begin{theorem} \label{conditions}
	If $P\supset 5\Delta$, then the three assumptions hold.
\end{theorem}

There exist Newton diagrams $P\subset \mathbb Z^2$ for which the Assumptions do not hold: the trivial $P=\Delta$ for which $C^\vee$ is a point gives an example. We know that Assumption 2 fails for some other polygons. We classify them in Section \ref{42}. This classification describes the polygons for which formula \eqref{mninflections} hold if the inflection points are counted with multiplicities.

\subsection{Classical results on Newton polytopes}
The following classical results are particular cases of the general BKK principle that we use in the paper.

\begin{theorem} [Kouchnirenko, Bernstein \cite{bernstein}]\label{kouchn}
	Let $f$ and $g$ be any square-free Laurent polynomials with $\newt(f) = P$, $\newt(g) = Q$. They define curves $C_1$ and $C_2$ in the torus. Then the index of intersection of their closures in the toric variety $X_{P}$ equals the mixed volume  \[\ol C_1 \circ \ol C_2 = \vol(P,Q)=V(P+Q)-V(P)-V(Q).\]
	
	Assume that $C_ 1$ and $C_2$ have no common components. Then the number of intersection points of $C_1$ and $C_2$ (counted with multiplicities) is $\vol(P,Q) - \sum_{p\in \ol C_1 \setminus C_1} \ol C_1 \circ_p \ol C_2$ where $\circ_p$ denotes the intersection number of the curves at point $p$.
\end{theorem}
\begin{remark}
	One may replace the toric variety $X_P$ by any $X_R$ with the following property: if $P^\gamma$ and $Q^\gamma$ are both edges, then $R^\gamma$ is an edge as well.
\end{remark} 

\begin{corollary} \label{discriminant}
	Under the conditions of Theorem \ref{kouchn}, assume that for each nonzero $\gamma \in (\mathbb R^2)^*$ the system $f^\gamma = g^\gamma = 0$ does not have solutions in $\torus^2$. Then the number of common zeroes of $f$ and $g$ (counted with multiplicities) in $\torus^2$ equals $V(P,Q)$.
	
	Here $f^\gamma$ denotes the sum of all the monomials of $f$ corresponding to the points of $P^\gamma$.
\end{corollary}

\begin{theorem}[Kouchnirenko, Bernstein, Khovanskii \cite{khovanskii}] \label{khov}
	Let $P$ be a lattice polygon and let $f$  be a generic polynomial in $\CC^P$. Then the genus of the curve $\{z \in \torus^2 \mid f(z)=0\}$ equals the number of lattice points that lie strictly inside $P$.
	
	Equivalently, the closure of this curve in the toric variety $X_P$ has Euler characteristic $-2V(P) + \sum_Q \len(Q)$ where $Q$ in the sum runs over all edges of $P$.
\end{theorem}

%%%%%%%%%%%%%%%%%%%%%%%%%%%%%%%%%%%%%%%%%%%%%

\section{Inflection points}
	In this section we prove Theorem \ref{inflection}. From now on let $P$ be a lattice polygon and $f$ be a generic Laurent polynomial in $\CC^P$. By $C$ we denote the curve $\{(x,y) \in \torus^2 \mid f(x,y) = 0\}$.
	
	The inflection points of $C$ are its intersection points with the hessian curve $H = \{\hess f = 0\} \subset \torus^2$ where 
	
	\begin{equation}
	\hess f = \begin{vmatrix}
		f_{xx} & f_{xy} & f_x \\
		f_{xy} & f_{yy} & f_y \\
		f_x & f_y & 0
	\end{vmatrix}
\end{equation}.

\begin{proposition} \label{newthess}
	Assume that the vertices of $P$ have large positive coordinates. Then $\newt (x^2 y^2 \hess f) = 3P$.
\end{proposition}
\begin{proof}
	It is enough to verify this fact for the monomials corresponding to the vertices of $P$. It is easy to see that the hessian of a monomial $ax^\alpha y^\beta$ is $a^3\alpha\beta(\alpha+\beta) x^{3\alpha -2}y^{3\beta - 2}$.
\end{proof}

In the rest of this section we still assume with no loss in generality that the vertices of $P$ have large positive coordinates.

Consider the toric variety $X_P$. In order to compute the number of intersection points of the curves $C$ and $H$ we do the following:
\begin{itemize}
	\item Compute the intersection number of the toric compactifications $\overline C$ and $\overline H$ of these curves in $X_P$. (Proposition \ref{intind})
	\item Compute the intersection multiplicity of $\overline C$ and $\overline H$ at each point of the 1-dimensional orbits of $X_P$. (Propositions \ref{genericdirection}, \ref{downdirection} and \ref{updirection})
	\item Use Theorem \ref{kouchn}. 
\end{itemize}

\begin{proposition} \label{intind}
	The intersection number of the curves $\ol C$ and $\ol H$ in $X_P$ is $\ol C \circ \ol H = 6 \vol(P)$.
\end{proposition}
\begin{proof}
	The statement follows directly from the Kouchnirenko-Bernstein theorem: 
	\begin{equation*}
		\ol C \circ \ol H = \vol(\newt(f),\newt(\hess f)) = \vol(P,3P) = 6 \vol (P).
	\end{equation*}
\end{proof}

%\begin{lemma}
%	Suppose that $\newt(g)$ is a segment not parallel to the lines $\{x=0\}$ , $\{y=0\}$, or $x+y=0$. Then the curve $\{g=0\}$ does not have inflection points.
%\end{lemma}

\begin{proposition} \label{genericdirection}
	Suppose that a covector $\gamma$ is not parallel to the covectors denoted by the arrows. Then the curves $\ol C$ and $\ol H$ do not intersect on the orbit $X_{P^\gamma}$.
\end{proposition}

%In particular, this proposition proves thorem \ref{inflection} in the case when $P$ does not have edges.
\begin{proof}
	For generic $f$ the curve $\ol C$ does not pass through the 0-dimensional orbits of $X_P$. So now let $\gamma$ be a covector for which the edge $P^\gamma$ satisfies the conditions of the proposition. 
	
	%We need to show that the curves $\{f^v = 0\}$ and $\{(\hess f)_v = 0\}$ do not have common points in $\torus^2$. 
	
	It is easy to see that $(\hess f)^\gamma = \hess (f^\gamma)$ since these two polynomials consist of the same monomials. So, by Corollary \ref{discriminant}, our problem is reduced to the fact that the system of equations $f^\gamma = \hess(f^\gamma) = 0$ does not have solutions. Suppose it does, i.e. $\{f^\gamma =0\}$ has inflection points. Note that the Newton polygons of the polynomials in the system are parallel segments, so their mixed volume is 0 and the system has no isolated solutions. This means that $\{f^\gamma =0\}$ has a nonisolated inflection point, hence, $\{f^\gamma =0\}$ contains a line and thus $\newt (f)^\gamma$ contains a Newton polygon of a line. It can only happen if $\newt (f)^\gamma$ is parallel to $\{x=0\}$, $\{y=0\}$ or $\{x + y = 0\}$ i.e. if $\gamma$ is parallel to a covector denoted by an arrow.

\end{proof}

\begin{proposition} \label{downdirection}
	Let $f \in \CC ^ P$ be a generic polynomial. 
	
	Suppose that Assumption 2 holds for the given $P$. Then for any $p \in \ol C \cap X_{P^\da}$ the local intersection number of $\ol C$ and $\ol H$ equals $\ol C \circ_p \ol H = 2$.
	
	In general, for any $P$, the inequality $\ol C \circ_p \ol H \ge 2$ holds.
\end{proposition}

\begin{proof}
First let us introduce some notation. By definition, for any complex $\alpha, \beta$ and natural $\gamma, \delta$ put
\newcommand{\ab}{{\alpha \beta}} 
	\begin{equation*}
	\partial^\ab_{x^\gamma y^\delta} f = x^{\gamma - \alpha} y^{\delta - \beta} {\partial^{\gamma + \delta} \over \partial x^\gamma \partial y^\delta} (x^\alpha y^\beta f)
\end{equation*}

The idea to consider such operators in a similar context first appeared in \cite{derivativetrick}. 

 It is easy to see that $\partial^\ab_{x^\gamma y^\delta} f$ is also a Laurent polynomial with Newton polygon contained in $P$, for instance, $\partial^\ab_{xx} f = \alpha (\alpha -1) f + 2\alpha x f_x + x^2 f_{xx}$. Now denote
%	\begin{equation*}
%	\hess^\ab (f) = \begin{vmatrix}
%		\alpha (\alpha -1) f + 2\alpha x f_x + x^2 f_{xx} & \alpha \beta f + \beta x f_x + \alpha y f_y + xy f_{xy} & \alpha f + x f_x \\
%		\alpha \beta f + \beta x f_x + \alpha y f_y + xy f_{xy} & \beta (\beta - 1) f + 2\beta y f_y + y^2 f_{yy} & \beta f + y f_y \\
%		\alpha f + x f_x & \beta f + y f_y & f
%	\end{vmatrix}.
%\end{equation*}
	\begin{equation*}
	\hess^\ab (f) = \begin{vmatrix}
		\partial^\ab_{xx} f & \partial^\ab_{xy} f & \partial^\ab_{x} f \\
		\partial^\ab_{yx} f & \partial^\ab_{yy} f & \partial^\ab_{y} f \\
		\partial^\ab_{x} f & \partial^\ab_{y} f &  0 
	\end{vmatrix}.
\end{equation*}
\newcommand{\OO}{\mathcal O} 

%Choose $\alpha_0$ and $\beta_0$ such that $\newt (x^{\alpha_0}y^{\beta_0}f)_\da \subset Ox$. Denote $\tilde f = x^{\alpha_0}y^{\beta_0}f$ and $\tilde P = \newt \tilde f$.

It is easy to see that the Newton polygons of all the entries of the matrix $\hess^\ab(f)$ are contained in $P$. Thus we may say that the polynomial $\hess^\ab(f)$ is a section of the bundle $\OO_{X_P}(P)^{\otimes 3} = \OO_{X_P}(3P)$. By $H^\ab$ denote the divisor of $X_P$ defined by this section. Note that $H^\ab$ in general non-trivially depends on $\alpha$ and $\beta$.

Now, the proof proceeds in the following four steps.

(I) The intersection of $H^\ab$ with $\ol C$ is contained in the finite set $\mathrm{(infl) \cup (\ol C \setminus C)}$ consisting of the inflections of $C$ and the intersection points of $\ol C$ and the 1-dimensional orbits of $X_P$. This set does not depend on $\alpha$ and $\beta$. 

Consequently, neither does the index $i = H^\ab \circ_p \ol C$.

(II) The number $H^{0,0}\circ_p \ol C$ obviously equals the sought one.

(III) For any lattice point $(\alpha_0,\beta_0)\in P^\da$, \begin{equation*}
	H^{-\alpha_0,-\beta_0} = 2X_{P^\da} + [\hess(x^{-\alpha_0}y^{-\beta_0}f)=0].
\end{equation*}
Thus, the desired index $i$ equals \begin{equation}
	2 + [\hess(x^{-\alpha_0}
y^{-\beta_0}f)
=0] \circ_p \ol C. \label{desired}
\end{equation}

(IV) The latter intersection number is clearly nonnegative. It vanishes under Assumption 2.

\textbf{Proof of (I).}
It can easily be checked that \begin{equation}
	\hess^\ab(f) = x^{2 - 3\alpha} y^{2 - 3\beta} \hess (x^\alpha y^\beta f) \label{hessab}
\end{equation} for any integers $\alpha$ and $\beta$. Let $\alpha$ and $\beta$ be any complex numbers. Then both sides of \eqref{hessab} can be defined on a neighborhood of any point of $\torus^2$, and \eqref{hessab} still holds. 

It now follows that the intersection $\ol C \cap H^\ab$ is contained in $\mathrm{(infl)} \cup (\ol C \setminus C)$. Indeed, each intersection point of $C\subset \torus^2$ and $H^\ab$ must be an inflection point of $\{x^\alpha y^\beta f = 0\} = C$.

The global index of intersection of $\ol C$ and $H^\ab$ does not depend on $\alpha$ and $\beta$ since $H^\ab$ are divisors of the same line bundle $\OO_{X_P}(3P)$. Thus, neither do the local indices at the points of $\mathrm{(infl)} \cup (\ol C \setminus C)$.

%Thus, all the curves of the two-parametric family $\{H^\ab \mid \alpha, \beta \in \CC\}$ intersect $\ol C$ only in the points of the finite set $I \cup (\ol C \setminus C)$. This means that for each $q \in I \cup (\ol C \setminus C)$ the local index of intersection $\ol C \circ_q H^\ab$ does not depend on $\alpha$ and $\beta$.

\textbf{Proof of (III).}
Now choose a lattice point $(\alpha_0, \beta_0)$ on $P^\da$. The monomial $x^{\alpha_0}y^{\beta_0}$ is a section of $\OO_{X_P}(P)$ that does not vanish on the chart $U = \torus^2 \cup X_{P^\da}\cong\torus \times \CC$. Thus $x^{\alpha_0}y^{\beta_0}$ forms a basis in each fiber. The coordinate expression of the section $f$ in this basis is the function $\tilde f = x^{- \alpha_0}y^{- \beta_0}f $. The closure $\ol C\subset U$ coincides with $\{\tilde f = 0\}$.

Similarly, $x^{3\alpha_0}y^{3\beta_0}$ is a section of $\OO_{X_P}(3P)$ that does not vanish on $U$.

 The coordinate expression of $\hess^{-\alpha_0, -\beta_0}(f)$ is \[\tilde h^{-\alpha_0, -\beta_0} =  x^{(2 + 3\alpha_0) - 3\alpha_0} y^{(2 + 3\beta_0) - 3\beta_0} \cdot \hess(x^{-\alpha_0}y^{-\beta_0}f) = x^2 y^2 \hess {\tilde f}.\]
 
 Hence,\begin{equation*}
 	 H^{-\alpha_0,-\beta_0} \cap U = [x^2y^2\hess(\tilde f)] = 2X_{P^\da} + [\hess(x^{-\alpha_0}y^{-\beta_0}f)=0].
 \end{equation*}
 
 Taking the intersection indices of with $\ol C$ both sides of the last equality, we obtain \eqref{desired}.
 
 \textbf{Proof of (IV).}
  Assumption 2 states that $p$ is not an inflection point of $\ol C \cap U = \{\tilde f = 0\}$, that is, $\ol C \circ_p \{\hess(\tilde f) = 0\} = 0$.

 \end{proof}

\begin{lemma} \label{vertical}
	 Suppose that Assumption 3 holds. Then the number of vertical tangents to $C$ is $2\vol(P) - \len(P^\da) - \len(P^\ua)$.
\end{lemma}
\begin{proof}
	
	Consider the projection $p\colon C \rightarrow \torus$, $(x,y) \mapsto x$. Suppose that a general point has $k$ preimages. The points with $k-1$ preimages correspond to $X_{P^\da} \cap \ol C$ (there are $\len(P^\da)$ of such points), to $X_{P^\ua} \cap \ol C$ ($\len(P^\ua)$ ones) and to the $n$ points with vertical tangents. By Assumption 3, there are no points with less than $k-1$ preimages. Now the Euler characteristic of $C$ is $e(C)=k\cdot e\torus - n - \len(P^\da) - \len(P^\ua)$. By the Kouchnirenko-Bernstein-Khovanskii theorem, $e(C) = -2\vol(P)$.
	
	In fact, we could apply the Riemann-Hurwitz formula, so Assumption 3 can be omitted here if we compute tangents with multiplicities.
\end{proof}

\begin{proposition} \label{updirection}
	Suppose that $f \in \CC ^ P$ is a generic polynomial. Then for any $p \in \ol C \cap X_{P^\ua}$ the local intersection number of $\ol C$ and $\ol H$ equals $\ol C \circ_p \ol H=1$.
\end{proposition}
\begin{proof} 
	First, we change the coordinates: $\hat x = x$, $\hat y = 1/y$. The Newton polygon of $f(\hat x, \hat y)$ is $\hat P$ which is symmetric to $P$ in the line $Ox$. \begin{align*}
		\hess(f(x,y)) &= 	
		\begin{vmatrix}
		f_{xx} & -\hat y^2 f_{x\hat y} & f_x \\
		-\hat y^2 f_{x\hat y} & 2 \hat y^3 f_{\hat y} + \hat y^4 f_{\hat y \hat y}& -\hat y^2 f_{\hat y} \\
		f_x & -\hat y^2 f_{\hat y} & 0
	\end{vmatrix} = \\
		&= \hat y^4 \begin{vmatrix}
			f_{xx} &  f_{x\hat y} & f_x \\
			f_{x\hat y} & \hat f_{\hat y \hat y}& f_{\hat y} \\
			f_x & f_{\hat y} & 0
		\end{vmatrix} + 2\hat y^3 f_{\hat y} f_x^2 =\\
		&= \hat y^4 \hess(f(x,\hat y)) + 2 \hat y^3 f_{\hat y} f_x^2.
	\end{align*}

In the new coordinates $p \in X_{\hat P^\da}$. Thus, it follows from the previous proposition that $\ol{\{\hess(f(x,\hat y))=0\}} \circ_p \ol {\{f=0\}} \ge 2$. Now it is enough to show that \begin{equation}
	\ol{\{2\hat y^3 f_{\hat y} f_x^2=0\}} \circ_p \ol {\{f=0\}} = 1. \label{remains}
\end{equation}.

 Note that for generic $f$, the curve $\ol{\{f_x=0\}}$ does not pass through $p$, so the identity \eqref{remains} is equivalent to \begin{equation}
 	\ol{\{f_{\hat y} =0\}} \circ_p \ol {\{f=0\}} = 1. \label{remrem}
 \end{equation}
 
This may be proved, for example, in the same way as the previous proposition. Another proof is as follows.

The equality \eqref{remrem} is equivalent to Lemma \ref{vertical}. Indeed, by Kouchnirenko-Bernstein, the system $f = f_y = 0$ has \[2\vol(P)-\sum_{\stackrel{\gamma \in (\mathbb R^2)^*}{p\in X_{P^\gamma}}}\ol{\{f=0\}}\circ_p\ol{\{f_y=0\}}\] solutions. It is easy to show that the terms of the sum are nonzero only if $\gamma \in \{\da,\ua\}$. The solutions of the system are exactly the points of $C$ with vertical tangents.
\end{proof}

\begin{lemma} \label{rotation}
	The transformation $r \colon \torus^2 \rightarrow \torus^2$, $x \mapsto 1/y$, $y \mapsto x/y$ sends lines to lines; vertical lines to horizontal lines, horizontal lines to lines passing through 0 and lines passing through 0 to vertical lines; inflection points and bitangents of $C$ to inflection points and bitangents to $r(C)$.
\end{lemma}
\begin{proof}
	The torus $\torus^2$ can be embedded into $\mathbb{CP}^2$ by $(x,y) \mapsto (x : y : 1)$, so $r$ becomes the projective transformation $(x:y:z) \mapsto (z:x:y)$.
\end{proof}

Suppose that we prove a statement regarding $P^\da$ and $P^\ua$. Then this lemma promises that the statement holds as well if $\ua$ and $\da$ are replaced by $\swa$ and $\nea$ or by $\ra$ and $\la$ respectively. Now we are ready to conclude the proof of Theorem \ref{inflection}.

\begin{proof}[End of the proof of theorem \ref{inflection}]
	By Theorem \ref{kouchn}, the number of the common points of $C$ and $H$ is $6\vol(P) - \sum_{p\in \ol C \setminus C} \ol C \circ_p \ol H$. Each point of $\ol C \setminus C$ is contained in some $1$-dimensional orbit of $X_P$. There are $\len P^\gamma$ common points of $\ol C$ and $X_{P^\gamma}$. We have shown that \begin{itemize}
		\item $\ol C \circ_p \ol H = 2$ for $p \in X_{P^\gamma}$ if $\gamma\in \da,\nea,\la$ (Proposition \ref{downdirection}),
		\item $\ol C \circ_p \ol H = 1$ if $\gamma\in\ua,\swa,\ra$ (Proposition \ref{updirection}),
		\item $\ol C \circ_p \ol H = 0$ otherwise (Proposition \ref{genericdirection}). \end{itemize} This finishes the proof. 
\end{proof}

\section{Bitangents}
\subsection{The tropical fan of the dual curve}
Recall that a line is nondegenerate if its equation can be written in the form $Ax+By+1=0$ with nonzero $A$, $B$, so its Newton polygon is the standard triangle $\Delta$. There are three kinds of degenerate lines: vertical, horizontal and passing through 0.

The natural map $v \colon C \rightarrow C^\vee$ sends the point $(x,y)\in C$ to the point $(a,b) \in \torus^2$ for which the line $ax + by + 1 = 0$ is tangent to $C$ at $(x,y)$. It is defined on all of $C$ except for the points $p$ in which $T_p C$ is degenerate since the equation defining $T_p C$ can not be written in form $ax + by + 1 = 0$. By $D$ denote the set of all such points.

Thus $v \colon C \rightarrow C^\vee$ is defined on $C \setminus D$. Consider the compactification $X_{ \Delta - P}$ of the torus containing $C^\vee$. Here $\Delta$ denotes the standard triangle with vertices $(0,0)$, $(1,0)$, $(0,1)$. By $\Delta - P$ we denote the Minkowski sum of $\Delta$ and the polygon centrally symmetric to $P$. Now in order to compute $\trop C^\vee$ we prove that the conditions of Definition \ref{trop} are satisfied for $C^\vee$ and $X_{\Delta - P}$. We also need to show that under Assumptions 1-3, the toric compactification of $C^\vee$ does not have singularities on 1-dimensional orbits. At first, we need some lemmas.

 At first, we need two simple lemmas.

\begin{lemma} \label{triangulation}
	Let $P$ be a lattice polygon for which the face $P^\da$ is an edge that is contained at the line $\{y = y_0\}$. Then $P$ contains at least one lattice point on the line $\{y = y_0 + 1 \}$.
\end{lemma}

\begin{proof}
	Consider any triangulation of $P$ into lattice triangles of area $1/2$. Choose two points $A$ and $B$ on $P^\da$ with $\len (AB) = 1$. They are vertices of a certain triangle $ABC$ of the triangulation. Since $\vol (ABC) = 1/2$, the ordinate of $C$ is $y_0 + 1$.
\end{proof}
\begin{lemma} \label{param}
	Let $f$ be a generic polynomial in $\CC^P$, $p \in X_{P^\da} \cap \overline{C}$, then there exists a paramitrization of $C$ in a neighbourhood of $p$ given by the formula 
	\begin{align*}
		x(t) &= a_0 + a_1 t + o(t)\\
		y(t) &= t
	\end{align*}
	such that $\lim_{t \ra 0} (x,y) = p$ and $a_0, a_1 \ne 0$.
\end{lemma}

\begin{proof}
	Suppose that $P^\da \subset \{y=0\}$.  Then $\torus^2 \cup X_{P\da}$ is isomorphic to $\torus\times\CC$. The number $a_0$ is just the $x$-coordinate of $p$. A generic curve $\{f=0\}$ intersects the orbit $X_{P^\da}$ transversally which means $f_x(p)\ne 0$ which proves the existence of parametrisation, possibly with $a_1 = 0$ which means that the tangent to $\{f=0\}$ at $p$ is vertical. However, by Lemma \ref{triangulation}, we may add a monomial $\alpha x^\beta y$ to $f$ such that the tangent to $ \{f+\alpha x^\beta y=0\}$ at $p$ is not vertical.
\end{proof}

\begin{proposition} \label{dualinfty}
	The map $v$ can be extended to the map $v \colon \ol C \rightarrow X_{\Delta - P}$. The image of $\ol C$ has no common points with the $0$-dimensional orbits of $X_{\Delta - P}$. 
	
	If Assumption 3 holds, then all the singularities of $\ol{C^\vee}$ are contained in $\torus^2$.
\end{proposition}
\begin{proof}
	Since $C=\{f=0\}$ we may calculate $v$ explicitly: \[v(x,y) = {-(f_x(x,y),f_y(x,y)) \over xf_x(x,y)  + yf_y (x,y)}.\]
	
	Consider a parameterization $x(t),y(t)$ of the curve $C$. It can be seen as a $\CC(\!(t)\!)$-valued point of $C$. By the implicit function theorem,  \begin{equation} \label{dualparam}
		v(x(t),y(t)) = {(\dot y,-\dot x) \over -\dot y x + \dot x y}.
	\end{equation}
	
	Denote $\val (at^\alpha + o(t^\alpha),bt^\beta + o(t^\beta)) = (\alpha, \beta)$ for nonzero $a$, $b$. Let $Q$ be any lattice polygon. It is a standard fact that the point $\lim_{t \ra 0} (x(t),y(t))$ lies on the orbit $X_{Q^{-\!\val(x(t),y(t))}}$ of the compactification $X_Q$ of $\torus^2$. 
	
	Now consider any point $p$ of the closure $\ol C \subset X_P$. Let us study the image of the neighbourhood of $p$ under $v$. There are several cases.
	
	\paragraph{Case 0. \label{simplerr}} Let $p \in \torus^2$ and the tangent $T_p C$ be a nondegenerate line. Then $v(p)$ is already defined.
	\paragraph{Case 1. } Let $p \in \torus^2$ and $T_p C$ be a degenerate line. Suppose that the tangent $T_p C$ is a vertical line. Assumption 3 states that $p$ is not an inflection point, so $C$ can be parameterized in a neighborhood of $p$ by $x = a + bt^2 + \dots$, $y=c + t$. It is easy to compute that \[v(x(t),y(t)) = \left(-{1\over a} + {2bc \over a^2} t + \dots, 2bc t + \dots\right).\] This means that $v(x(t),y(t))$ tends to the orbit $X_{(\Delta-P)^\da}$, intersects it transversally and that $v(\ol C)$ does not have a singularity at $\lim_{t \ra 0} v(x(t),y(t))$. Define $v(p)$ as this limit.
	
	The other points of $D \subset C$ are tracked by Lemma \ref{rotation}.
	The rest of the cases deal with the points on 1-dimensional orbits of $X_P$.
	
	\paragraph{Case 2.} \label{par:frominfty}	The rest of the cases deal with the points on 1-dimensional orbits of $X_P$.
	Suppose $p \in \ol C \cap X_{P^\da}$. By Lemma \ref{param}, a neighborhood of $p$ in $C$ is parameterized by $x = a + bt + \dots$, $y=t$. Then \[v(x(t),y(t)) = (-1/a + \dots, b/a + \dots).\] This means that we may define $v(p) = (-1/a, b/a) \in \torus^2$.

	\paragraph{Case 3.}
	Suppose $p \in \ol C \cap X_{P^\ua}$. By Lemma \ref{param} (under the coordinate change $y \mapsto 1/y$), a neighborhood of $p$ in $C$ is parameterized by $x = a + bt + \dots$, $y=1/t$. Then \[v(x(t),y(t)) = \left (-{1 \over a} + {2b \over a^2} t + \dots, -{b \over a} t^2 + \dots \right).\] This means that $v(p)$ lies on $X_{(\Delta-P)^\da}$ and $v(\ol C)$ is \textbf{tangent} to the orbit $X_{(\Delta-P)^\da}$ at $v(p)$. 
	
	By Assumption 3, all images of the  points discussed in the Cases 1 and 3 are distinct, which means that $v(\ol C)$ does not have singularities on $X_{(\Delta-P)^\da}$. Indeed, the cusps on $X_{(\Delta - P)^\da}$ correspond to the inflection points of $C$ in which the tangent is vertical; the vertical bitangents and vertical asymptotes that are tangent to $C$ at a finite point correspond to double points of $\ol{C^\vee}$. Note that a vertical asymptote is a vertical line that has a common point with $\ol C$ on $X_{P^\ua}$.
	
	\paragraph{Case 4.}
	Suppose $p \in \ol C \cap X_{P^\gamma}$ with $\gamma = (\alpha,\beta)$ not parallel to the covectors denoted by arrows. A neighborhood of $p$ in $C$ is parameterized by $x = at^{-\alpha} + \dots$, $y=bt^{-\beta} + \dots$. Then \[v(x(t),y(t)) = \left({\beta \over a(\alpha + \beta)}t^\alpha + \dots, {-\alpha \over b(\alpha+\beta)}t^\beta + \dots\right).\] This means that $v(p)$ lies on $X_{(\Delta-P)^{-\gamma}}$ and $v(\ol C)$ is transversal to the $1$-dimensional orbit $X_{(\Delta-P)^{-\gamma}}$ at $v(p)$. Note that the fractions $a^\beta/b^\alpha$ are distinct for all the points $p\in  \ol C \cap X_{P^\gamma}$ so the images $v(p)$ are distinct as well.
\end{proof}
We will use the following result which is a particular case of \cite[Theorem 1.1]{dualvartrop}. We also give a self-contained proof of this in the Appendix. All the technical work needed is already done in Proposition \ref{dualinfty}.
\begin{theorem} \label{dualtrop}
	The tropical fan $\trop C^\vee$ consists of:
	\begin{itemize}
		\item the rays of $\trop C$ except for $\cone(\da)$, $\cone(\ua)$, $\cone(\nea)$, $\cone(\swa)$, $\cone(\la)$, $\cone(\ra)$. They are taken with the same weights but with the opposite directions;
		\item the ray $\cone(\da)$ taken with weight $2V(P) - \len(P^\da) + \len(P^\ua)$;
		\item the rays $\cone(\nea)$ and $\cone(\la)$ with analogical weights.
	\end{itemize}
\end{theorem}

\begin{example} \label{theexample}
	Consider the curve $C = \{xy + y + 1 =0\} \subset \torus^2$. Its tropical fan consists of the rays $\cone (\ua)$, $\cone (\la)$ and $\cone(1,-1)$ of weight 1. The area of $\newt (C)$ is $1/2$. They are depicted on Figure \ref{fig:tropcvee}.

Thus $\trop C^\vee$ contains the ray $\cone(-1,1)$ of weight 1 as the opposite of $\cone(1,-1)$, the ray $\cone(\da)$ of weight $2\cdot1/2-0+1 = 2$, the ray $\cone(\nea)$ of weight $2\cdot1/2-0+0 = 1$, and $\cone(\la)$ of weight $2\cdot1/2-1+0 = 0$.

However, the last one does not affect the tropical fan. The Newton polytope $P^\vee$ of the dual curve can now easily be recovered as we can see on the figure.

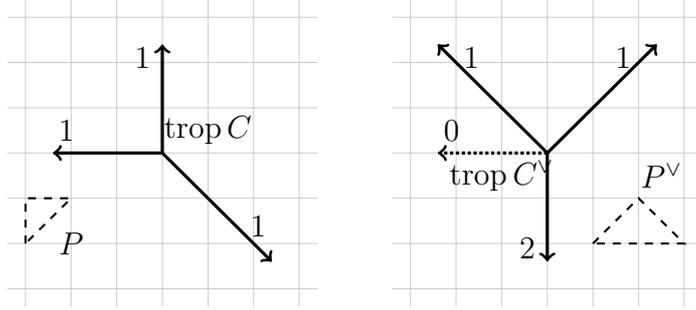
\begin{figure}  
	\centering
	\begin{tikzpicture}[scale=0.6]
			\draw[step=1,black!20,very thin] (-3.4,-3.4) grid (3.4,3.4);
			\draw[black,dashed, thick] (-3, -2) -- ++(0,1) -- ++(1,0) -- cycle;
			\draw[->, black, very thick] (0, 0) --  node[very near end, left] {1} ++(0,2.4);
			\draw[->, black, very thick] (0, 0) -- node[very near end, above] {1} ++(-2.4,0);
			\draw[->, black, very thick] (0, 0) -- node[very near end, above] {1} ++(2.4,-2.4);	
			\node at (-2,-2) {$P$};
			\node at (1,0.5) {$\trop C$};
		\end{tikzpicture} \hspace{20pt}
		\begin{tikzpicture}[scale=0.6]
			\draw[step=1,black!20,very thin] (-3.4,-3.4) grid (3.4,3.4);
			\draw[black, dashed, thick] (1, -2) -- ++(1,1) -- ++(1,-1) -- cycle;
			\draw[->, black, very thick] (0, 0) --  node[very near end, left] {2} ++(0,-2.4);
			\draw[->, black, very thick] (0, 0) --  node[very near end, left] {1} ++(2.4,2.4);
			\draw[->, black, very thick] (0, 0) -- node[very near end, right] {1} ++(-2.4,2.4);
			\draw[->, black, densely dotted, very thick] (0, 0) -- node[very near end, above] {0} ++(-2.4,0);
			\node at (2.5,-0.5) {$P^\vee$};
			\node at (-1,-0.5) {$\trop C^\vee$};
		\end{tikzpicture}
	\caption{An example of computation of $\trop C^\vee$ }
	\label{fig:tropcvee} 
\end{figure}
\end{example}

\newcommand{\vedg}{\begin{tikzpicture}
		\tikzstyle{bla}=[circle,draw=black,fill=black, thin,
		inner sep=0pt,minimum size=2pt]
		\draw[black, thick] (0, 0pt) -- (0,7pt);
		\node at (0,0pt) [bla] {};
		\node at (0,7pt) [bla] {};
\end{tikzpicture}} 
\newcommand{\hedg}{\begin{tikzpicture}
		\tikzstyle{bla}=[circle,draw=black,fill=black, thin,
		inner sep=0pt,minimum size=2pt]
		\draw[black, thick] (0, 0pt) -- (7pt,0pt);
		\node at (0,0pt) [bla] {};
		\node at (7pt,0pt) [bla] {};
\end{tikzpicture}} 
\newcommand{\dedg}{\begin{tikzpicture}
		\tikzstyle{bla}=[circle,draw=black,fill=black, thin,
		inner sep=0pt,minimum size=2pt]
		\draw[black, thick] (7pt, 0pt) -- (0,7pt);
		\node at (7pt,0pt) [bla] {};
		\node at (0,7pt) [bla] {};
\end{tikzpicture}} 

\begin{proposition} \label{dualvolume}	By $\vedg$, $\hedg$ and $\dedg$ denote the edges of $\Delta$. Set $S = \vol(P)$. For any $\gamma \in \{\da, \la, \nea\}$ denote $l^\gamma=\len(P^\gamma)$.
	The Newton polygon of $C^\vee$ is the polygon $P^\vee$ whose normal fan is $\trop(C^\vee)$. Its area is
	\begin{align*} 
		\vol(P^\vee) &= 2S^2 + 2S \vol(\Delta, -P) - 2S(l^\da+l^\nea + l^\la)+S-\\
		&  - (l^\da \vol(P,\hedg) + l^\nea \vol(P,\dedg) + l^\la \vol(P,\vedg)) + (l^\da l^\nea + l^\nea l^\la + l^\la l^\da)
	\end{align*}

Here $-P$ is the polygon centrally symmetric to $P$.
\end{proposition}
\begin{proof}
 By Theorem \ref{dualtrop}, \begin{equation}
		P^\vee = 2\vol(P)\Delta + (-P) - l^\da \hedg - l^\nea \dedg - l^\la \vedg.
	\end{equation}.
Here $(-P)$ is the polygon centrally symmetric to $P$. The addition and subtraction are understood in the sense of the group of virtual convex polygons. 

Now, expanding the mixed volume $\vol(P^\vee,P^\vee)$ as a symmetric bilinear form, we obtain the desired result.
\end{proof}
\subsection{Computation of the Euler characteristics}

Now we conclude the proof of Theorem \ref{bitangents}.

\begin{proof}[Proof of theorem \ref{bitangents}]
	First, note that every point $p \in \ol{C^\vee}$ which is not a node has exactly 1 preimage under $v$. The nodes have 2, which means
	\begin{equation}
		e(\overline C) - e(\overline{C^\vee}) = \#\text{(bitang)}. \label{euler1}
	\end{equation}

	 Assumption 1 states that the singularities of $\ol{C^\vee}$ are nodes and cusps; Assumptions 2 and 3 state that  $\ol{C^\vee}_{sing}\subset\torus^2$ and thus the singularities of ${C^\vee}$ correspond bijectively to the inflection points and bitangents of $C$. 
	 
	 By $f^\vee$ denote the defining Laurent polynomial of $C^\vee$. Consider a generic (thus, smooth) closed curve $C_1 \subset X_{P^\vee}$ given by $\ol{\{f^\vee(x,y) - \varepsilon x^\alpha y^\beta = 0\}}$ with $(\alpha,\beta) \in P^\vee$ and generic $\varepsilon \in \CC$. The Euler characteristics of $C_1$ and $\ol{C^\vee}$ differ by the sum of Milnor numbers of the singularities of $\ol{C^\vee}$. By Proposition \ref{dualinfty}, $\ol{C^\vee}$ has no singularities on the 1-dimensional orbits of $X_{P^\vee}$. We obtain the following equality.
	\begin{equation}
		e(\overline{C^\vee}) - e(C_1) = 2\#\text{(infl)} + \#\text{(bitang)}. \label{euler2}
	\end{equation}
	Adding \eqref{euler1} and \eqref{euler2} up and counting the Euler characteristics of $\ol C$ and $C_1$ by Theorem \ref{khov} we obtain the statement of the theorem.
	
	Indeed, both $C$ and $C_1^\vee$ are generic, and it is easy to compute the lattice perimeter of $P^\vee$. The lattice lengths of the edges of $P^\vee$ are just the weights of the corresponding rays of $\trop C^\vee$:
		\begin{align*}
		&2\#(\mathrm{infl}) + 2\#(\mathrm{bitang}) = e(\overline C) - e(C_1) = \\ &= -2\vol(P) + \sum_{\gamma\in \mathbb Z^2} \len{P^\gamma} + 2\vol(P^\vee)-\sum_{\gamma\in \mathbb Z^2} \len{(P^\vee)^\gamma}=\\
		 &=-8\vol(P) + 2\vol(P^\vee) + 2\sum_{\gamma\in\da,\nea,\la}\len(P^\gamma).
	\end{align*}

	Substituting the number of inflection points, we obtain the desired formula.
\end{proof}

\section{Possible multiplcities of tangency} \label{assumptionsproof}

In this section we list conditions on $P$ under which the assumptions we made hold. Theorem \ref{notritangents} essentially shows the method we use; all the other proofs are similar. In fact, one may prove stronger results, see Theorem \ref{notri}.

\subsection{Assumption 1 holds}

First, we want to show that generically $C^\vee$ has no singularities other than nodes and cusps. 

We call a line nondegenerate if its equation can be written in the form $ax+by+1=0$ with nonzero $a$, $b$. This is equivalent to the fact that its Newton diagram is the standard triangle $\Delta$. Otherwise, the line is degenerate. There are three kinds of such lines: vertical, horizontal and passing through 0.

Let $l$ be a nondegenerate line $\{ax + by +1 =0\}$. Suppose that the point $l^\vee=(a,b)$ lies on $C^\vee$.  Then the closure of $l$ in $X_P$ is tangent to $\ol C$ (i.e. intersect it with multiplicity $\ge2$). The tangency points belong either to $\torus^2$, or to one of the orbits $X_{P^\da}$, $X_{P^\la}$, $X_{P^\nea}$ since these are the only orbits a nondegenerate line can intersect (c.f. \nameref{par:frominfty} in the proof of Proposition \ref{dualinfty}). By $k$ denote the number of tangency points of $C$ and $l$. Suppose that $l^\vee \in C^\vee_{sing}$. Then one of the following 12 possibilities of the mutual position of $C$ and $l$ is realized.

\paragraph{Possibility 1} $k \ge 3$. Then $C^\vee$ has $k \ge 3$ branches at $l^\vee$.

A: (Excluded by Th. \ref{notritangents}) Every tangency point belongs to $\torus^2$ 

B: (Excl. by Th. \ref{nobitangents2}) One tangency point does not 

C: (Excl. by Th. \ref{nobitangents3}) Two or three do not

\paragraph{Possibility 2} $k = 2$ and one of the tangency points is an inflection one. Then $C^\vee$ has 2 branches at $l^\vee$ at least one of which is not smooth.

A: (Excl. by Th. \ref{nobitangents}) Both tangency points belong to $\torus^2$ 

B: (Excl. by Th. \ref{nobitangents2}) One does not. 

C: (Excl. by Th. \ref{nobitangents3}) Both do not.

\paragraph{Possibility 3} $k = 1$ and the intersection multiplicity of $C$ and $l$ at the tangency point is $m \ge 4$. Then $l^\vee$ is a degenerate cusp (given locally by $x = t^{m-1}$, $y = t^m + o(t^m)$ in some coordinates).

A: (Excl. by Th. \ref{noinflections}) The tangency point $p$ belongs to $\torus^2$. 

B: (Excl. by Th. \ref{infleinfty}) It does not. 

\paragraph{Possibility 4} $k = 2$ and the multiplicity of intersection at both of the tangency points is 2. Then $l^\vee$ is a node.

A: The tangency points belongs to $\torus^2$  and thus $l$ is a bitangent.

B: (Excl. by Th. \ref{nobitangents2}) At least one does not. Then we do not call $l$ a bitangent. 

\paragraph{Possibility 5} $k = 1$ and the multiplicity of intersection at the tangency point is 3. Then $l^\vee$ is a cusp.

A: The tangency point $p$ belongs to $\torus^2$. It is an inflection point.

B: (Excl. by Th. \ref{infleinfty}) It does not. In this case by definition we do not call $p$ an inflection point. 

\paragraph{}

In this section we prove theorems of the kind ”If $P$ is large enough, then a possibility can not be realized 
for a generic curve $C$”. We write the number of every such theorem in the brackets after the corresponding 
possibility. Since possibilities 4A and 5A are the only ones that are not excluded by such theorems for generic curves with large enough Newton polygons, we conclude that 4A and 5A  are the only possible cases for generic curves with large enough Newton polygons. 

These two cases are exactly the simple
bitangent and simple inflection.

The method that we use in the proof of Theorems \ref{notritangents}-\ref{infleinfty} can be regarded as a partial answer to the question from \cite[Remark 3.23]{affchar} for discriminants of pairs of polynomials $(f,l) \in \CC^P \times \CC^\Delta$ in the sense of \cite{sysdiscr}.

\begin{theorem} \label{notritangents}
	Let $Q$ be one of the diagrams on Figure \ref{fig:tritangdiagrams}. Suppose that a set $P\subset \mathbb Z^2$ contains $Q$. Then a generic curve in $\CC^P$ has no tritangents (nondegenerate lines tangent to the curve at three distinct points).
\end{theorem}
\begin{figure}  
	\centering
	\begin{tikzpicture}[scale=0.5]
	\tikzstyle{ver}=[circle,draw=black,fill=black!40, thin,
	inner sep=0pt,minimum size=0.2cm]
	\draw[step=1,black!20,very thin] (-0.4,-0.4) grid (5.4,5.4);
	\draw[black,dashed, thick] (0, 0) -- (0,5) -- (5,0) -- cycle;
	\node at (0,0) [ver] {};
	\node at (1,0) [ver] {};
	\node at (2,0) [ver] {};
	\node at (3,0) [ver] {};
	\node at (4,0) [ver] {};
	\node at (5,0) [ver] {};
\end{tikzpicture}
\begin{tikzpicture}[scale=0.5]
\initver
\draw[step=1,black!20,very thin] (-0.4,-0.4) grid (5.4,5.4);
\draw[black,dashed, thick] (0, 0) -- (0,5) -- (5,0) -- cycle;
\node at (0,4) [ver] {};
\node at (1,3) [ver] {};
\node at (2,2) [ver] {};
\node at (2,1) [ver] {};
\node at (2,0) [ver] {};
\node at (3,2) [ver] {};
\end{tikzpicture}
\begin{tikzpicture}[scale=0.5]
	\tikzstyle{ver}=[circle,draw=black,fill=black!40, thin,
	inner sep=0pt,minimum size=0.2cm]
	\draw[step=1,black!20,very thin] (-0.4,-0.4) grid (5.4,5.4);
	\draw[black,dashed, thick] (0, 0) -- (0,5) -- (5,0) -- cycle;
	\node at (0,0) [ver] {};
	\node at (1,0) [ver] {};
	\node at (1,1) [ver] {};
	\node at (2,1) [ver] {};
	\node at (2,2) [ver] {};
	\node at (3,2) [ver] {};
\end{tikzpicture}
\begin{tikzpicture}[scale=0.5]
	\tikzstyle{ver}=[circle,draw=black,fill=black!40, thin,
	inner sep=0pt,minimum size=0.2cm]
	\draw[step=1,black!20,very thin] (-0.4,-0.4) grid (5.4,5.4);
	\draw[black,dashed, thick] (0, 0) -- (0,5) -- (5,0) -- cycle;
	\node at (0,0) [ver] {};
	\node at (1,0) [ver] {};
	\node at (2,0) [ver] {};
	\node at (0,3) [ver] {};
	\node at (0,4) [ver] {};
	\node at (0,5) [ver] {};
\end{tikzpicture}
	\caption{A few diagrams of class $\mathcal Q_6$. If $P$ contains one of them, a generic curve from $\CC^P$ has no tritangents.}
	\label{fig:tritangdiagrams} 
\end{figure}
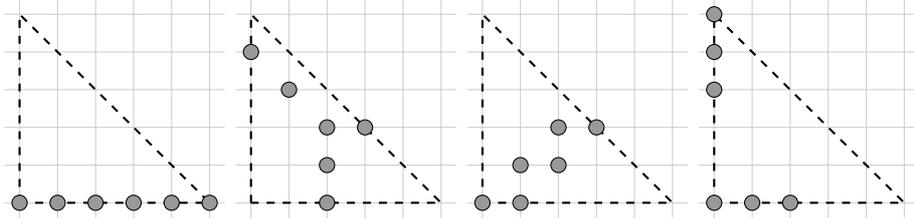
\begin{proof}
	First, we prove that there is a nondegenerate line $l$ and three points on it with the folowing property: for any 6-point subdiagram $Q\subset P$ there is a curve from $\CC^Q$ tangent to $l$ at that points.
	
	Consider the (not closed) variety $V \subset \CC^P \times \torus^5$ consisting of the pairs $(f,(A,B,X_1,X_2,X_3))$ such that:
	\begin{itemize}
	\item the curve $\{f=0\}$ is tangent to the line $\{Ax+By+1=0\}$ at the points with abscissas $X_1, X_2, X_3$;
	\item the numbers $X_1$, $X_2$, $X_3$ are distinct.
	\end{itemize}

	By $p_1 \colon V \rightarrow \CC^P$ and $p_2 \colon V \rightarrow \torus^5$ denote the canonical projections. Suppose that a generic curve in $\CC^P$ has a tritangent. Then $p_1(V) \subset \CC^P$ is dense and thus $\dim V \ge \left| P \right|$. Now it follows that there exists a point $X \in \torus^5$ such that $\dim p_2^{-1}(X) \ge \left| P \right| -5$. 
	
	Let $W_X$ be the vector subspace of $\CC^P$  consisting of the curves that have a tritangent encoded by $X$. We have shown that $W_X$ has dimension at least $\left| P \right| -5$. 
	
	The vector subspace $\CC^Q \subset \CC^P$ has dimension 6, and thus contains a nonzero polynomial from $W_X$, i.e. a curve having a tritangent.
	
	Now, for each diagram $Q$ on Figure \ref{fig:tritangdiagrams} we show that there are no such curves in $\CC^Q$. Suppose that the curve $\{f=0\}$ with $f \in \CC^Q$ has a tritangent. Then either the degree of $f$ is at least 6, or $\{f=0\}$ contains a nondegenerate line. Since the diagram $Q$ can fit into $5\Delta$, the degree is not more than 5. 
	
	Suppose that $f$ is divisible by a linear polynomial with Newton polygon $\Delta$. By $R \subset Q$ denote the Newton diagram of $f$. Then $\Delta$ is a Minkowski summand of $\conv(R)$. However, it is easy to check that this is not the case for every subset $R$ of every $Q$ listed on the figure.
\end{proof}

Let us introduce the following 
\begin{definition}
	A diagram $Q\subset \mathbb Z^2$ is of class $\mathcal Q_d$ if the following conditions hold.
	\begin{itemize}
	\item $Q$ consists of $d$ points;
	\item $Q$ can be shifted into a subset of $(d-1) \Delta$;
	\item for any $f \in \CC^Q$ the curve $\{f=0\}$ does not contain nondegenerate lines.
	\end{itemize}
	
	Suppose that $Q \subset (d-1)\Delta$ is a $d$-point diagram and for any subset $R \subset Q$ the triangle $\Delta$ is not a Minkowski summand of $\conv R$. Then $Q$ is of class $\mathcal Q_d$, as we have seen in the proof of Theorem \ref{notritangents}.
	
	Thus, for example, a segment $AB$ parallel to any of the lines $\{x=0\}$, $\{y=0\}$ or $\{x+y=0\}$ is of class $\mathcal Q_{\len (AB) + 1}$.
\end{definition}
\begin{theorem} \label{nobitangents}
	Let $C$ be a generic curve with Newton diagram $P$. 
	
	Suppose that $P$ contains a diagram $Q$ of class $\mathcal Q_5$. Let $l$ be a nondegenerate line tangent to $C$ at two distinct points $p$ and $q$. Then none of $p$, $q$ is an inflection point of $C$, i.e. $C \circ_p l = C \circ_q l = 2$.
	
\end{theorem}
%\begin{figure}  
%	\centering
%	\input{pic/bitangdiagrams.tex}
%	\caption{ }
%	\label{fig:bitangdiagrams} 
%\end{figure}

\begin{proof}
	The proof is essentially the same as for Theorem \ref{notritangents}.
	
	 In this case let $V\subset \CC^P \times \torus^4$ be the set of all pairs $(f,(A,B,X_1,X_2))$ with $X_1 \ne X_2$ for which $\{f=0\}$ intersects the line $\{Ax+By+1=0\}$ at the points of ordinates $X_1$ and $X_2$ with multiplicity at least 3 and at least 2 respectively. 
	 
	 The numbers 6 and 5 in the proof of \ref{notritangents} should be replaced by 5 and 4 respectively.
\end{proof}

\begin{theorem} \label{noinflections}
	Suppose that $P$ contains a diagram $Q$ of class $\mathcal Q_4$. Let $l$ be a nondegenerate line tangent to $C$ at some point $p$. Then $C \circ_p l \le 3$.
\end{theorem}
%\begin{figure}  
%	\centering
%	\input{pic/infldiagrams.tex}
%	\caption{ }
%	\label{fig:infldiagrams} 
%\end{figure}
\begin{proof}
Again, the proof is essentially the same as for Theorem \ref{notritangents}.
\end{proof}

\begin{theorem} \label{nobitangents2}
	Let $P$ be a diagram. Suppose that one of the following conditions hold.
	\begin{itemize}
		\item[(1)] The face $P^\da$ is a vertex.
		\item[(2)] There exists a diagram $Q$ of class $\mathcal Q_4$ such that $Q \subset P$ and $Q^\da \subset P^\da$. (For instance, this condition holds if $P$ is a polygon such that $\len P^\da\ge3$. In this case take $Q \subset P^\da$.)
		\item[(3)] The face $P^\da$ is contained in the line $y=y_0$ and the polygon $P$ contains two lattice points of the same ordinate $y_0 + \beta$ with $\beta\ge 2$.
	\end{itemize} 
	Then for a generic curve $C$ from $\CC^P$ there are no nondegenerate lines tangent to $\ol C$ at any two points $p \in \torus^2$ and $q \in X_{P^\da}$.
\end{theorem}
\begin{proof}
	If $P^\da$ is a vertex, then $P^\da$ is a 0-dimensional orbit, so $C$ does not pass through $q$.
	
	So, assume that $P^\da$ is an edge and condition (2) holds. The union $\torus^2 \cup X_{P^\da}$ is isomorphic to $\torus \times \CC$. Without loss of generality, suppose that $P^\da$ is contained in the line $\{y=0\}$. In this case $\ol C \subset\torus \times \CC$ is given by $f(x,y) = 0$. 
	
	For $V$ take the set of all pairs $(f, (A,B,X))$ such that the line $\{(x,y) \in \torus \times \CC \mid Ax + By +1 =0\}$ at the point of ordinate $X$ and at $(-1/A,0)$.
	
	Suppose that the statement of the theorem does not hold. Applying the same arguments as before, we get that there exists $g \in \CC^Q$ for which $C' = \{(x,y) \in \torus \times \CC \mid g(x,y)=0\}$ is tangent to a nondegenerate line at $p \in \torus^2$ and $q \in X_{P^\da}$. 
	
	Since $Q^\da \subset P^\da$, the curve $C'$ coinsides with the intersection of a projective curve of degree 3 and $\torus \times \CC$ (but not the union of a cubic curve and $\{y^k=0\}$). The curve $C'$ cannot have bitangents since it does not contain lines.
	
	Now, assume that condition (3) holds and $P^\da$ is still an edge. 
	
	Consider a 4-point diagram $Q \subset P$ consisting of 2 neighboring points on $P^\da$ and 2 neighboring points of ordinate $y_0 + \beta$. Again, we may assume that $P^\da$ is contained in the horizontal axis, so $y_0=0$.
	
	Define $V$, $g$, $C'$, the same way as in the case $(2)$ for the new diagram $Q$. So, $g$ has the form $x^\alpha (a_0 + a_1 x + a_2 x^{\gamma}y^{\beta} + a_3 x^{\gamma + 1}y^{\beta})$. 
	
	If both of the numbers $a_0$ and $a_1$ are nonzero, then $C'$ contains a unique smooth point $q$ with vanishing $y$-coordinate. By the condition $\beta \ge 2$ the tangent $T_q C$ is vertical, so a nondegenerate line cannot be tangent to $C$ at $q$.
	
	If exactly one of the numbers $a_0$ and $a_1$ is zero, then $C'$ has no points with vanishing $y$-coordinate.
	
	If both are zero, then $C'$ is the union of a vertical line and the line $\{y^\beta=0\}$. In this case it is clear that there are no nondegenerate lines tangent to $C'$ at a point in $\torus^2$.
\end{proof}

\begin{theorem} \label{nobitangents3}
	Suppose that the Newton polygon $P$ is not contained in a segment and $P \ne \Delta$. Then there are no nondegenerate lines tangent to $\ol C$ at any two points $p \in X_{P^\la}$ and $q \in X_{P^\da}$.
\end{theorem}

\begin{proof}
	The proof is analogical to the previous one. If $P^\la$ or $P^\da$ is a point, then $C$ does not pass through it. So, assume that these faces are edges.
		
	Without loss of generality, suppose that $P^\da \subset \{y=0\}$ and $P^\la \subset \{x=0\}$. 
	
	Let $V$ be the set of all pairs $(f, (A,B))$ such that the nondegenerate line $\{(x,y) \in \CC \times \CC \mid Ax + By +1 =0\}$ is tangent to $C$ at $(-1/A,0)$ and at $(0,-1/B)$.
	
	We need to find a 3-point diagram $Q$ for which any curve from $\CC^Q$ has no bitangents of the described kind.
	
	Suppose that $(0,0) \notin P$ or $\len (P^\la) \ge 2$. Then we may choose $Q$ consisting of the points $(0,a)$, $(b,0)$, $(b+1,0)$ with $a \ge 2$.
	
	Any curve $C'$ in $\CC^Q$ is given either by $y = (\beta x^b + \gamma x^{b+1})^{1/a}$ or by $x = \alpha$. In both cases the tangent to $C'$ at the point that lies on $\{y=0\}$ is vertical.
	
	Suppose that $(0,0) \in P$ and $(1,1) \in P$. We may choose $Q$ consisting of the points $(0,0)$, $(1,1)$ and $(0,1)$. It is a trivial check that a curve supported at $Q$ has no nondegenerate bitangents at all.
	
	Now, if $(0,0) \in P$, $(1,1) \notin P$, $\len P^\la = \len P^\da = 1$, then $P=\Delta$ which concludes the proof. 
\end{proof}

\subsection{Assumption 2 holds} \label{42}

The following example shows that the assumption is nontrivial.
\begin{example} \label{cubic}
	Suppose that $P$ is the triangle with vertices $(0,3)$, $(1,0)$, $(2,0)$. Then $\ol{i(C)}$ is a cubic curve that has inflection points at $(0:0:1)$ and $(1:0:0)$, so it has a third inflection point on the line $\{y=0\}$. This is an inflection point at infinity.
	
	Similarly, consider the triangle $P$ with vertices $(1,0)$, $(2,0)$ and $(1-k, 1+2k)$, $k \ge 0$ and any $f\in \CC^P$. Let $l(x,y)$ be the sum of the three monomials of $f$ that correspond to $(1,0)$, $(2,0)$ and $(1,1)$. The polynomial $l$ defines a line that intersects $\ol{f=0}$ on $X_{P^\da}$ at some point $p$ with multiplicity at least 3:
	\begin{align*}
		\ol{\{f=0\}} \circ_p \ol{\{l=0\}} = \ol{\{f-l=0\}} \circ_p X_{P^\da} \ge 3,
	\end{align*}
	since $f-l$ is a polynomial divisible by $y^3$.
\end{example}

%\begin{theorem}  Suppose that $P$ is a lattice polygon, that $f\in \CC^{P}$ is a generic Laurent polynomial and that $\len P^\da \ge 2$ or $P^\da$ is a vertex. Then the curve $C$ does not have inflection points at infinity that lie on $X_{P^\da}$. Equivalently, the intersection multiplicity of $\ol{\{f = 0\}}$ and any line at any $p\in X_{P^\da}$ is at most 2.
	
\begin{definition}
	We call the triangles discussed in the example \textit{thin}.
\end{definition} We will show that such triangles are essentially the only polygons for which Assumption 2 fails.
\begin{theorem} \label{infleinfty} Suppose that $P$ is not a thin triangle.  If $f\in \CC^{P}$ is a generic Laurent polynomial, then the curve $C$ has no inflection points at infinity that lie on $X_{P^\da}$. 
	
	Equivalently, the intersection multiplicity of $\ol{\{f = 0\}}$ and any line at any $p\in X_{P^\da}$ is at most 2.
	
\end{theorem}

\begin{proof}\label{proofinfleinfty}
	
	If $P^\da$ is a vertex, then $C$ has niinflection points on $X_{P^\da}$. So, without loss of generality, assume that $P^\da$ is an edge contained in the horizontal coordinate axis.
	
	For $V$ take the set of pairs $(f,(A,B))$ such that the curve the intersection index of $\{(x,y) \in \torus \times \CC \mid f = 0\}$ and the line $\{Ax + By +1 = 0\}$ at $(-1/A,0)$ is at least 3.
	
	Similarly to the previous theorems, it is enough to construct a 3-point diagram $Q \subset P$ such that for any $g\in \CC^Q$, any point $(x_0,0) \in \{g=0\}$ and a line $l$ the inequality $\{g=0\} \circ_{(x_0,0)} l< 3$ holds. By Lemma \ref{param}, the tangent lines to the points on $X_{P^\da}$ are not vertical, so we may assume that $l$ is nondegenerate.
	
	Now suppose that there exists a diagram $Q \subset P$ consisting of 2 neighboring points of the edge $P^\da$ and a point with ordinate 1 such that $Q$ is a triangle not equal to $\Delta$. 
	
	The polynomials from $\CC^Q$ are (up to lattice translations) of the form $f(x,y) = ay - b x^\alpha - c x^{\alpha + 1}$ with $\alpha \ne 0$. If $a=0$, then this polynomial defines a vertical line and is not tangent to nondegenerate lines. Suppose that $a\ne 0$ and $\{f = 0\}$ has an inflection point $(x_0,0)$ with $x_0 \ne 0$. This would imply  $ b\alpha(\alpha-1) x_0^{\alpha - 2} + c \alpha (\alpha+1) x_0^{\alpha -1} = 0$ which contradicts the equality $b x_0^{\alpha-2} + c x_0^{\alpha -1} = x_0^{-2}f(x_0,0)=0$.
	
	So, by Lemma \ref{triangulation}, we may assume that $\{(x,y)\in \mathbb Z^2 \mid 0 \le y \le 1\} \cap P = \Delta$.
	
	Suppose that $P$ contains a point $(c,2)$ for some $c\in \mathbb Z$. Consider the diagram $Q = \{(0,0),(1,0),(c,2)\} \subset P$.
	
	The curves supported at $Q$ are of the form $y = x^k\sqrt{a+bx}$ or of the form $y^2=0$. In the first case, the tangent at $(-b/a, 0)$ is vertical, so a nondegenerate line intersects the curve with multiplicity $1<3$. In the second case, any nondegenerate line intersects the curve with multiplicity $2$.
	
	Thus, we have proved the theorem statement for all $P$ that do not satisfy the condition $\{(x,y)\in \mathbb Z^2 \mid 0 \le y \le 2\} \cap P = \Delta$. It is easy to see that the polygons satisfying this condition are thin.
\end{proof}

\paragraph{Remark.} We have obtained that $C$ has an inflection point on $X_{P^\da}$ if and only if $P$ is thin. By Lemma \ref{rotation}, this means that Assumption 2 is not satisfied if and only if one of the polygons $P$, $r(P)$ or $r^2(P)$ is thin (see Figure \ref{fig:assumption2}). 

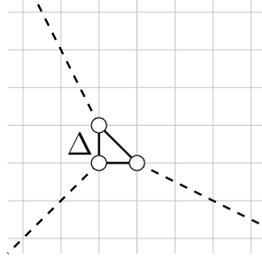
\begin{figure}  
	\centering
	\begin{tikzpicture}[scale=0.5]
	\tikzstyle{ver}=[circle,draw=black,fill=black!0, thin,
	inner sep=0pt,minimum size=0.2cm]
	\draw[step=1,black!20,very thin] (1-0.4,1-0.4) grid (7.4,7.4);
	\draw[black, thick] (3, 3) -- (3,4) -- (4,3) -- cycle;
	\draw[black, dashed, thick] (3, 3) -- (1-0.4,1-0.4);
	\draw[black, dashed, thick] (3, 4) -- ++(-1.7,3.4);
	\draw[black, dashed, thick] (4, 3) -- ++(3.4,-1.7);
	\node at (3,3) [ver] {};
	\node at (3,4) [ver] {};
	\node at (4,3) [ver] {};
	\node at (2.5,3.5) {$\Delta$};
\end{tikzpicture}
	\caption{Assumption 2 fails if and only if $P$ is the convex hull of $\Delta$ and a point on one of the dashed rays.}
	\label{fig:assumption2} 
\end{figure}

\subsection{Assumption 3 holds}
 In this subsection let $P \subset \mathbb Z^2$ be a Newton diagram and let $C$ be a generic curve supported at $P$. By $\pi \colon \mathbb R^2 \rightarrow \mathbb R$ denote the projection $(x,y) \mapsto y$.
Note that in general, $\pi(\conv(P)) \cap \mathbb Z \ne \pi(P)$.
\begin{theorem} \label{novert1}
 If $\pi(P)$ contains the four points $\{a,a+1,a+2,a+3\}$ for some $a\in \mathbb Z$, then $C$ has no vertical bitangents.
\end{theorem}

\begin{proof}
	Consider the variety $V$ consisting of the pairs $(f,(X_0,Y_1,Y_2))$ such that the line $\{x=X_0\}$ is tangent to $\{f=0\}$ at $(X_0,Y_1)$ and $(X_0,Y_2)$.  
	
	Let $Q\subset P$ be a 4-point diagram with $\pi(Q)=\{a,a+1,a+2,a+3\}$. It is enough to show that any curve $\{f=0\}$ supported at $Q$ has no vertical bitangents.
	
	Indeed, if the curve is tangent to $\{x=X_0\}$ at two distinct points, then $f(X_0,y) \in \CC[y]$ is a degree 3 polynomial that has two multiple roots, hence, $f(X_0,y)$ is identically 0.  
\end{proof}

\begin{theorem} \label{novert2}
	If $\pi(P)$ contains at least three points, then the tangents at the inflection points of $C$ are not vertical. 
\end{theorem}
\begin{proof}
	Let $Q\subset P$ be a 3-point diagram consisting of points with distinct $y$-coordinates. 
	
	Suppose that $f\in \CC^Q$. If a vertical line $\{x=X_0\}$ intersects the curve $\{f(x,y)=0\}$ with multiplicity 3 at some point, then the polynomial $f(X_0,y)$ has a triple root. However, $f(X_0,y)$ can be expressed as the sum of three monomials, so it cannot have roots of multiplicity 3 (see e.g. \cite{bivariate}, the remark after theorem 1.1). For example, this is evident for $Q$ such that $\pi(Q) = \{a,a+1,a+2\}$ for some $a$ since $f(X_0,y)$ has degree 2.
\end{proof}

\begin{theorem} \label{novert3}
	If $\pi(P)$ contains at least three points or $P^\ua$ is a vertex, then $C$ has no vertical tangents that are also asymptotes. 
\end{theorem}

\begin{proof}
	
	If $P^\ua$ is a vertex, then $C$ has no vertical asymptotes.
	
	Assume that $P^\ua$ is an edge contained in the line $\{y=0\}$. Denote $\tilde f(x,y)=f(x,1/y)$ and $\tilde P=\newt \tilde f$. So, $\tilde P^\da \subset \{y=0\}$.
	
	The vertical asymptotes are exactly the lines $\{x=X_0\}$ that have a common point with $\tilde f(x,y)=0$ on the line $\{y=0\}\subset \torus^2 \cup X_{\tilde P^\da} \cong \torus \times \CC$, i.e. $\tilde f(X_0,0)=0$.
	
	So, we show that a generic curve supported at $\tilde P$ has no vertical tangents that intersect the closure of the curve $\{\tilde f = 0\}$ on $\{y=0\}$.

	Let $V$ be the set of all pairs $\{(\tilde f,(X_0,Y_0))\}$ such that $\{\tilde f=0\}$ is tangent to $\{x = X_0\}$ at $(X_0,Y_0)$ and $\tilde f(X_0,0)=0$.
	
	Choose a 3-point diagram $Q \subset \tilde P$ such that $\pi(Q) = \{0,a,b\}$. Suppose that $f = \alpha x^k + \beta x^l y^a + \gamma x^m y^b \in \CC^Q$. If $f(X_0,0)=0$, then $\alpha = 0$. But in this case it is easy to check that $f(X_0,y)$ can not have a double root at any $y=Y_0$ since $f(X_0,y)$ is the sum of two monomials in $y$.
\end{proof}

\subsection{The conclusion of the proof of Theorem \ref{conditions}}

If $P\supset 5\Delta$, then the conditions of Theorems \ref{notritangents}, \ref{nobitangents}, \ref{noinflections} are satisfied since $5\Delta$ contains diagrams of the type $\mathcal Q_6$, $\mathcal Q_5$ and $\mathcal Q_4$. It is easy to check that the conditions of Theorems \ref{nobitangents2}(3), \ref{nobitangents3}, \ref{infleinfty}, \ref{novert1}, \ref{novert2} and \ref{novert3} are also satisfied (possibly, under the map $r$ of Lemma \ref{rotation}).

\begin{example} \label{rectangle}
	 Let us show that the assumptions are satisfied for the $3 \times 4$ rectangle $P$.
	 
	 Theorems \ref{notritangents}, \ref{nobitangents} and \ref{nobitangents2} require subdiagrams of class $\mathcal Q_6$, $\mathcal Q_5$ and $\mathcal Q_4$ respectively. We have marked these subdiagrams on Figure \ref{fig:3by4}.
	 
	 The conditions of Theorems \ref{nobitangents2}(3), \ref{nobitangents3}, \ref{infleinfty}, \ref{novert2} and \ref{novert3} are clearly satisfied.
	 
	 Theorem \ref{novert1} asserts that a generic curve supported at $P$ has no horizontal bitangents or bitangents that pass through 0. Since $\pi(P)$ does not have 4 neighboring points, Theorem \ref{novert1} is not applicable to vertical bitangents. However, the $y$-degree of a curve supported at $P$ is 3, so it can not have vertical bitangents.  
	 
\end{example}

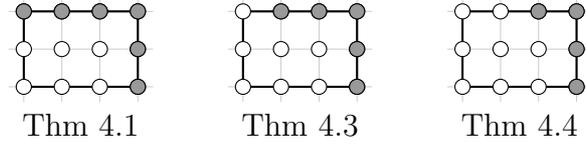
\begin{figure}  
	\centering
	\begin{tikzpicture}[scale=0.5]
	\tikzstyle{ver}=[circle,draw=black,fill=black!0, thin,
	inner sep=0pt,minimum size=0.2cm]
	\draw[step=1,black!20,very thin] (-0.4,-0.4) grid (3.4,2.4);
	\draw[black, thick] (0, 0) -- (3,0) -- (3,2) -- (0,2) -- cycle;
	\foreach \x in {0,...,3}
		\foreach \y in {0,...,2}
			\node at (\x,\y) [ver] {};
	\tikzstyle{ver}=[circle,draw=black,fill=black!40, thin,
	inner sep=0pt,minimum size=0.2cm]
	\node at (0,2) [ver] {};
	\node at (1,2) [ver] {};
	\node at (2,2) [ver] {};
	\node at (3,2) [ver] {};
	\node at (3,1) [ver] {};
	\node at (3,0) [ver] {};
	\node at (1.5,-1) {Thm \ref{notritangents}};
\end{tikzpicture} \hspace{20pt}
\begin{tikzpicture}[scale=0.5]
	\tikzstyle{ver}=[circle,draw=black,fill=black!0, thin,
	inner sep=0pt,minimum size=0.2cm]	
	\draw[step=1,black!20,very thin] (-0.4,-0.4) grid (3.4,2.4);
	\draw[black, thick] (0, 0) -- (3,0) -- (3,2) -- (0,2) -- cycle;
	\foreach \x in {0,...,3}
	\foreach \y in {0,...,2}
	\node at (\x,\y) [ver] {};
	\tikzstyle{ver}=[circle,draw=black,fill=black!40, thin,
	inner sep=0pt,minimum size=0.2cm]
	\node at (1,2) [ver] {};
	\node at (2,2) [ver] {};
	\node at (3,2) [ver] {};
	\node at (3,1) [ver] {};
	\node at (3,0) [ver] {};
	\node at (1.5,-1) {Thm \ref{nobitangents}};
\end{tikzpicture} \hspace{20pt}
\begin{tikzpicture}[scale=0.5]
	\tikzstyle{ver}=[circle,draw=black,fill=black!0, thin,
	inner sep=0pt,minimum size=0.2cm]
	\draw[step=1,black!20,very thin] (-0.4,-0.4) grid (3.4,2.4);
	\draw[black, thick] (0, 0) -- (3,0) -- (3,2) -- (0,2) -- cycle;
	\foreach \x in {0,...,3}
	\foreach \y in {0,...,2}
	\node at (\x,\y) [ver] {};
	\tikzstyle{ver}=[circle,draw=black,fill=black!40, thin,
	inner sep=0pt,minimum size=0.2cm]
	\node at (2,2) [ver] {};
	\node at (3,2) [ver] {};
	\node at (3,1) [ver] {};
	\node at (3,0) [ver] {};
	\node at (1.5,-1) {Thm \ref{noinflections}};
\end{tikzpicture}
	\caption{We use the gray subdiagrams to show that Assumption 1 holds for the rectangle.}
	\label{fig:3by4} 
\end{figure}

\section{Appendix: a direct computation of the tropical fan of the dual curve}

\begin{theorem}
	The tropical fan $\trop C^\vee$ consists of:
	\begin{itemize}
		\item the rays of $\trop C$ except for $\cone(\da)$, $\cone(\ua)$, $\cone(\nea)$, $\cone(\swa)$, $\cone(\la)$, $\cone(\ra)$. They are taken with the same weights but with the opposite directions;
		\item the ray $\cone(\da)$ taken with weight $2V(P) - \len(P^\da) + \len(P^\ua)$;
		\item the rays $\cone(\nea)$ and $\cone(\la)$ with analogical weights.
	\end{itemize}
\end{theorem}
\begin{proof}
	We use Definition \ref{trop}. We have already computed all the local indices of intersection in Proposition \ref{dualinfty}.
	
	Suppose that $\gamma$ is not parallel to the arrows. Then the indices are 1 for the $\len P^{-\gamma}$ points of $\ol{C^\vee}$ on $X_{(\Delta - P)^\gamma}$, so $\ol{C^\vee} \circ X_{(\Delta - P)^\gamma} = \len(P^{-\gamma})$.
	
	They are 1 for the images $v(p) \in X_{(\Delta - P)^\da}$ of the points that have vertical tangents. By Lemma \ref{vertical}, there are $2\vol(P) - \len(P^\da) - \len(P^\ua)$ such points. The indices are $2$ for the $\len(P^\ua)$ images of points that lie on $\ol C \cap X_{(\Delta-P)^\da}$. Thus, $\ol{C^\vee} \circ X_{(\Delta - P)^\da} = 2 \vol(P) - \len(P^\da) - \len(P^\ua) + 2\len(P^\ua)$.
	
	The intersection numbers with $X_{(\Delta-P)^\la}$ and $X_{(\Delta-P)^\nea}$ can now be found using Lemma \ref{rotation}.
\end{proof}

Now let us explain the casework of Proposition \ref{dualinfty} in a simple case.
\begin{example}
	Let us turn turn back to Example \ref{theexample}. We consider the curve $C = \{xy + y + 1 =0\} \subset \torus^2$.
	
	It is easy to compute that the dual curve is given by the equation $a^2 + 4ab -2a + 1 = 0$. Now we investigate the behavior of the dualization map $v$ as we did in Proposition \ref{dualinfty}, see Figure \ref{fig:hyperbola}.
	
	Case 1. By Lemma \ref{vertical}, the curve $C$ has no vertical or horizontal tangents and one tangent line that passes through 0. The tangency point is $p_1(-1/2,-2)$. 
	
	The curve $C^\vee$ tends asymptotically to the line $p_1^\vee=\{-a/2-2b+1=0\}$ and thus $\ol{C^\vee}$ passes through the point $q_1 = \ol{p_1^\vee} \cap X_{(P^\vee)^\nea}$. In fact, $q_1=v(p_1)$.
	
	Case 2. The curve $\ol C$ has a unique point $p_2$ on the orbit $X_{P^\la}$. The line $\{-x+y+1=0\}$ is tangent to $C$ at $p_2$ and thus the point $q_2=(-1,1)\in \torus^2$ on the dual plane is the image of $p_2$ under the duality map $v$.
	
	Case 3. If $p_3 = \ol C \cap X_{P^\ua}$, then the tangent at $p_3$ is given by $\{y+1=0\}$, so $\ol{C^\vee}$ passes through $v(p_3)=q_3=(1,0)\in X_{(P^\vee)^\da}$. 
	
	To see intuitively why $\ol{C^\vee}$ is tangent to the orbit $X_{(P^\vee)^\da}$, consider the usual embedding $i$ of $C$ into $\mathbb P^2$. Evidently, $\ol {i(C)}$ passes through $(0:1:0)$, so $C^\vee$ should be tangent to the corresponding line $\{b=0\}$ on the dual plane.
	
	Case 4. The point $p_4 = \ol C \cap X_{P^{(1,-1)}}$ maps to $q_4  = \ol {C^\vee} \cap X_{(P^\vee)^{(-1,1)}}$.
	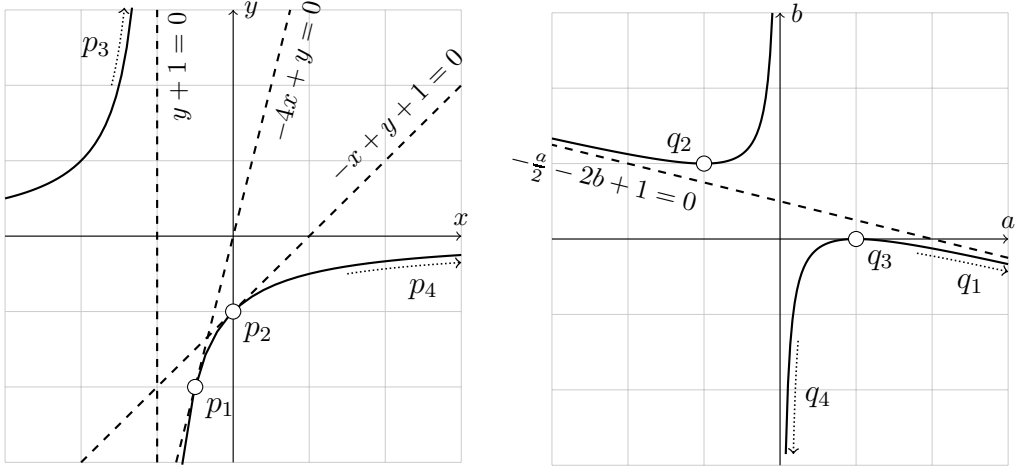
\begin{figure}  
		\centering
		\begin{tikzpicture}
	\draw[step=1,black!20,very thin] (-3,-3) grid (3,3);
	\draw[->] (-3,0) --  node[at end,above]{\footnotesize $x$} (3,0);
	\draw[->] (0,-3) -- node[at end,right]{\footnotesize $y$} (0,3);
	\draw[thick]  plot[variable=\x, domain=-3:-1.33] ({\x},{-1/(\x+1)})
		   plot[variable=\x, domain=-1+0.33:3] ({\x},{-1/(\x+1)});
	\draw[thick, dashed] (-3/4,-3) -- 
	node[very near end,sloped,below]{\footnotesize $-4x + y = 0$} (3/4,3);
	\draw[thick, dashed] (0-1*2,-1-1*2) -- node[very near end,sloped,above]{\footnotesize $-x + y + 1 = 0$} (0+1*3,-1+1*3);
	\draw[thick, dashed] (-1,-3) -- node[very near end,sloped,below]{\footnotesize $ y + 1 = 0$}(-1,3);
	\draw[fill=white] (-0.5,-2) circle (0.1) node[anchor=north west]{$p_1$};
	\draw[fill=white] (-0,-1) circle (0.1) node[anchor=north west]{$p_2$};
	
	\draw[->,semithick,densely dotted,yshift=-0.1cm]	plot[variable=\x, domain=1.5:3] ({\x},{-1/(\x+1)}) ;
	\node at (2.5, -0.7) {$p_4$};
	
	\draw[->,semithick,densely dotted,xshift=-0.1cm]	plot[variable=\x, domain=-1.5:-1.33] ({\x},{-1/(\x+1)}) ;
	\node at (-1.8, 2.5) {$p_3$};
\end{tikzpicture}
\begin{tikzpicture}
	\draw[step=1,black!20,very thin] (-3,-3) grid (3,3);
	\draw[->] (-3,0) -- node[at end,above]{\footnotesize $a$} (3,0);
	\draw[->] (0,-3) -- node[at end,right]{\footnotesize $b$}  (0,3);
	
	\draw[thick]  plot[variable=\x, domain=-3:-0.1,samples=100] ({\x},{-(\x-1)^2/(4*\x)})
	plot[variable=\x, domain=0.075:3,samples=100] ({\x},{-(\x-1)^2/(4*\x)});
	
	\draw[thick, dashed] (3,-1/4) -- node[very near end,sloped,below]{\footnotesize $ -\frac{a}{2} - 2b +1 =0$}(-3,5/4);
	
	\draw[fill=white] (1,0) circle (0.1) node[anchor=north west]{$q_3$};
	\draw[fill=white] (-1,1) circle (0.1) node[anchor=south east]{$q_2$};
	\draw[->,semithick,densely dotted,yshift=-0.1cm]	plot[domain=1.8:3] ({\x},{-(\x-1)^2/(4*\x)}) ;
	\node at (2.5,-0.6) {$q_1$};
	\draw[<-,semithick,densely dotted,xshift=0.1cm]	plot[domain=0.075:0.14] ({\x},{-(\x-1)^2/(4*\x)}) ;
	\node at (0.5,-2.1) {$q_4$};
\end{tikzpicture}
		\caption{The curves $C = \{xy + y + 1 = 0 \}$ and $C^\vee$. The point $p_i$ corresponds to Case $i$ of Proposition \ref{dualinfty}; $q_i = v(p_i)$.}
		\label{fig:hyperbola} 
	\end{figure}
\end{example}

\section{Appendix: a stronger version of Theorem \ref{notritangents}}
\begin{theorem} \label{notri}
	Let $R$ be any lattice parallelogram of area 1. Assume that $P$ contains $5R$. Then a generic curve in $\CC^P$ has no tritangents.
\end{theorem}
Consider a nondegenerate line $l\subset \torus^2$ and 3 points $p_1$, $p_2$ and $p_3$ on it such that for any 6-point $Q\subset P$ there is a curve $C_Q$ supported at $\CC^Q$ tangent to $l$ at the points. We have shown the existence of such a line while proving Theorem \ref{notritangents}.

 We will construct several (namely, eight) such $Q$'s for which the condition cannot be satisfied simultaneously.
 
 It follows that for any 6-point $Q_1,Q_2 \subset P$ the number of common points of $C_{Q_1}$ and $C_{Q_2}$ in $\torus^2$ counted with multiplicities is $C_{Q_1} \circ C_{Q_2} \ge 6$ since the curves are tangent to each other at three points (by definition, we put $C_{Q_1} \circ C_{Q_2} = \infty$ if $C_{Q_1}$ and $C_{Q_2}$ have a common component).

First, we need several lemmas.
\begin{lemma} \label{binomial}
	Let $S$ be any lattice segment with $\len(S) = 1$. Consider two curves supported at (i.e. having Newton diagram) $S$. Then they have no common tangent lines.
\end{lemma}
\begin{proof}
	This is trivial if the curves are lines.
	
	Otherwise, it is easy to compute that the dual curves are also distinct binomial curves supported at $S$, so have no common points. 
\end{proof}
\begin{lemma} \label{2on1}
	Let $S$ be any lattice segment with $\len(S) = 1$. Assume that $P$ contains $5S$. Then 2 of the points $x_1$, $x_2$ and $x_3$ lie on the same binomial curve with Newton diagram $S$.
\end{lemma}
\begin{proof}
	First, using an argument similar to Lemma \ref{triangulation}, choose a point $A\in P$ such that $\vol(\conv(A,5S)) = 5$. By $T$ denote the convex hull of $A$ and any 5 neighboring points on $5S$. Then the mixed volume $\vol(5S,T) = 5$. If $C_{5S}$ and $C_T$ have no common components, then, by Kouchnirenko-Bernstein, $C_{5S} \circ C_T \le 5$. 
	
	Hence, $C_{5S}$ and $C_{T}$ have a common component $\tilde C$. The Newton diagram of $\tilde C$ must be a segment contained in $5S$. By $f_T$ denote the defining polynomial of $C_T$ (note that, in general, $\newt (f_T) \ne T$). Then $\newt(\tilde C)$ is a Minkowski summand of $\newt(f_T)$, so $\newt(f_T)$ is actually a segment contained in $5S$. Its lattice length is at most $4$ by the definition of $T$.
	
	So, $C_T$ is the union of 4 binomial curves $C_1, \dots, C_4$ supported at $S$ (some of them may coincide).
	
	Assume that none of the curves $C_i$ contain two of the points $p_1$, $p_2$ and $p_3$ simultaneously. Then two of the binomial curves contain at most one of the points (say, $C_1$ and $C_2$). The line $l$ is tangent to $C_1$ and $C_2$ simultaneously, which contradicts Lemma \ref{binomial}.
	
	Note that if $C_1$ and $C_2$ coincide, then any line having at least 3 common point with $C_1$ is a tritangent of $C_T$.
\end{proof}

\begin{proof}[Proof of Theorem \ref{notri}]
	Choose the coordinates $(\tilde x, \tilde y)$ on $\torus^2$ such that $R$ becomes the standard square $[0,1]\times[0,1]$. In these coordinates $l = \{1 + a \tilde x^\alpha \tilde y^\beta + b \tilde x^\gamma \tilde y^\delta=0\}$ with \begin{equation*} \
		\begin{vmatrix}
		\alpha & \beta \\
		\gamma & \delta
	\end{vmatrix}=1.
	\end{equation*}

	The parallelogram $R$ contains 4 segments: its edges $E_1=[0,1]\times0$, $E_2 = 0 \times [0,1]$ and diagonals $D_1=[(0,1),(1,0)]$, $D_2$.
	
	The segments $5E_i$ and $5D_i$ satisfy the conditions of Lemma \ref{2on1}. This gives curves $C_1=\{\tilde x = a_1\}$, $C_2=\{\tilde y = a_2\}$, $C_3=\{\tilde y = a_3 \tilde x\}$ and $C_4=\{\tilde x\tilde y = a_4\}$, supported at $E_1$, $E_2$, $D_1$ and $D_2$ respectively. Each of the curves passes through some 2 of the points $p_1$, $p_2$, $p_3$.
	
	Note that each pair of the first 3 curves has exactly 1 common point. So, without loss of generality, we may assume that $p_3 = C_1 \cap C_2$,  $p_2 = C_1 \cap C_3$ and $p_1 = C_2 \cap C_3$. The coordinates of $p_3$ are $(a_1,a_2)$. Changing the coordinates $(\tilde x,\tilde y) \mapsto (a_1 \tilde x,a_2 \tilde y)$, we may assume that $p_3=(1,1)$, $p_2=(1,a_3)$ and $p_1=(a_3^{-1},1)$. 
	
	The curve $C_4$ may contain two of the points only if $a_3=-1$. 
	
	In fact, the proof above is based on the fact that $\vol(E_i,D_j) = \vol(E_1,E_2)=1$.
	
	Now suppose that $p_1, p_2$ and $p_3$ lie on $l$. Then the following system of equations in $a,b$ has a solution:
	
	\begin{equation} \label{finalsystem}
		\begin{cases}
		1 + a\cdot 1\cdot 1 + b\cdot 1 \cdot 1 = 0, \\
		1 + a (-1)^\alpha + b (-1)^\gamma = 0, \\
		1 + a (-1)^\beta + b (-1)^\delta = 0.
	\end{cases}
	\end{equation}

	If $\varepsilon \in \mathbb Z$,  denote $\varepsilon'=\varepsilon \:\mathrm{mod}\: 2\in\{0,1\}$. Since $\alpha \delta - \beta \gamma = 1$, we have $\alpha' \delta' - \beta' \gamma'=\pm1$.
	
	Hence, the determinant of the system is \[\begin{vmatrix}
		1 & 1 & 1\\
		1 & (-1)^\alpha & (-1)^\beta \\
		1 & (-1)^\gamma & (-1)^\delta
	\end{vmatrix}=
\begin{vmatrix}
	1 & 1 & 1\\
	0 & (-1)^\alpha-1 & (-1)^\beta-1 \\
	0 & (-1)^\gamma-1 & (-1)^\delta-1
\end{vmatrix}=
\begin{vmatrix}
	1 & 1 & 1\\
	0 & -2\alpha' & -2\beta' \\
	0 & -2\gamma' & -2\delta'
\end{vmatrix}= \pm 4.\] 
	
Thus, the system \eqref{finalsystem} has no solutions. This concludes the proof.
	
\end{proof}
\bibliographystyle{alpha}

\end{document}